\documentclass[oneside,english,reqno]{amsart}

\usepackage{algorithm}
\usepackage{algorithmic}
\usepackage{multirow}

\usepackage{enumitem}

\usepackage[T1]{fontenc}
\usepackage[latin9]{inputenc}
\usepackage{babel}
\usepackage{amsmath}
\usepackage{amsthm}
\usepackage{amssymb}

\usepackage{comment}

\usepackage{graphicx}

\usepackage{graphicx}
\usepackage{tikz}
\usetikzlibrary{shapes.misc}
\usepackage{float}

\usepackage[unicode=true,pdfusetitle,
bookmarks=true,bookmarksnumbered=false,bookmarksopen=false,
breaklinks=false,pdfborder={0 0 1},backref=false,hidelinks]
{hyperref}

\makeatletter

\numberwithin{equation}{section}
\numberwithin{figure}{section}

\makeatother

\newtheorem{theorem}{Theorem}[section]

\newtheorem{corollary}[theorem]{Corollary}

\newtheorem{definition}[theorem]{Definition}
\newtheorem{example}[theorem]{Example}

\newtheorem{lemma}{Lemma}[section]
\newtheorem{fact}{Fact}[section]

\newtheorem{proposition}[theorem]{Proposition}
\newtheorem{remark}[theorem]{Remark}

\newcommand{\hilbertX}{\mathcal{ X}}

\newcommand{\rank}{\text{rank}\,}

\newcommand{\spanc}{\overline{\operatorname{span}}}
\newcommand{\setM}{\mathcal{ M}}
\newcommand{\coneN}{\mathcal{ N}}

\begin{document}

\title[On Abadie Condition under RCRCQ+]{Abadie condition for infinite programming problems under Relaxed Constant Rank Constraint Qualification Plus
}

\author{
	Ewa M. Bednarczuk$^{1,2}$
}
\author{
	Krzysztof W. Le\'sniewski$^2$
}
\author{
Krzysztof E. Rutkowski$^{2,3}$ 
}
\thanks{$^1$ Warsaw University of Technology, 00-662 Warszawa, Koszykowa 75,		
}
\thanks{$^2$ Systems Research Institute of the Polish Academy of Sciences, 01-447, Warszawa, Newelska 6,
}
\thanks{$^3$ Cardinal Stefan Wyszy\'nski University, 
	Faculty of Mathematics and Natural Sciences. School of Exact Sciences, 01-815, Warszawa,  Dewajtis 5}

\begin{abstract}

We consider infinite programming problems with constraint sets defined by systems of infinite number of inequalities and equations given by continuously differentiable functions defined on Banach spaces.
In the approach proposed here we represent these systems with the help of coefficients in a given Schauder basis. We prove the Abadie condition  under the new infinite-dimensional Relaxed Constant Rank Constraint Qualification Plus and we discuss the existence of Lagrange multipliers. The main tools are: Rank Theorem and Ljusternik Theorem. 

\end{abstract}
\maketitle

\section{Introduction}

 Let $E$ be a Banach space and $f_0,\ g_i:\ E \rightarrow \mathbb{R}$, $i\in I_0\cup I_1$, $I_0\cap I_1=\emptyset$ are functions of class $C^1$.
 We consider the following \textit{infinite programming} problem
    \begin{align}\tag{$P_0$}\label{problem:P0} 
    \begin{aligned}
         &\operatorname{Minimize}_{x\in\mathcal{ F } }f_0(x)\\
        \mathcal{ F}=&\left\{ x \in E \ \big|\  \begin{array}{ll}
             g_i(x) = 0,&  i\in I_0,\\
             g_i(x) \leq 0,& i\in I_1
        \end{array}\right\},
     \end{aligned}
     \end{align}
i.e., both index sets $I_0,I_1$ can be infinite. 

This problem can be cast into a general framework. 
Let $F$ be a Banach space.  Let $(b_{i})_{i\in \mathbb{N}}$ be a 
basis of $F$ i.e. for all $x\in F$ there is a unique sequence of scalars $(x_n)$, $n\in \mathbb{N}$ s.t. $x=\sum\limits_{n=1}^\infty x_n b_n$. 

The following proposition establishes the existence of a biorthogonal system, i.e. existence of a sequence of associated linear functionals $(b_i^*)\in F^*$, $i\in \mathbb{N}$ s.t.
\begin{equation*}
	b_i^*( v_j) =\delta_{ij}= \left\{ \begin{array}{ll}
	1 & \text{if } i=j,\\
	0 & \text{if } i\neq j
	\end{array}\right.
	\end{equation*}
 \begin{proposition}
\label{fact:existence_dial_basis}(\cite[Theorem 1.1.3]{MR2192298})
	Let $X$ be a Banach space and let $(v_n)_{n\in \mathbb{N}}$ be a basis of $X$. 
	Then there is a sequence $(v_n^*)_{n\in \mathbb{N}} \in X^*$ such that 
%
%
%
%
	\begin{equation}\label{delta_1}
	v_i^*( v_j) =\delta_{ij}= \left\{ \begin{array}{ll}
	1 & \text{if } i=j,\\
	0 & \text{if } i\neq j
	\end{array}\right.
	\end{equation}
	and $x=\sum_{n\in \mathbb{N}} v_n^*(x)v_n$ for each $x \in X$. Such pair $(v_n,v_n^*)$ is called biorthogonal system.
 \end{proposition}

When $F$ is finite-dimensional ($\operatorname{dim} F= n$) we assume that $(b_{i})_{i\in \mathbb{N}}$ is a basis of $F$ (see \cite[Example (i), Chapter 6]{MR1831176}).  Let us note that any separable Hilbert space has a basis. 
 
 Let $K\subset F$ be a cone defined as
    \begin{equation}\label{cone:K}
        K:
        = \{ y \in F \mid y= \sum_{i\in I_0\cup I_1} b_i^*(y) b_i, \ b_i^*(y)= 0,\ i\in I_0,\ b_i^*(y)\leq 0, \ i\in I_1 \},
    \end{equation}
    where $I_0\cup I_1\subseteq  \mathbb{N} $, $I_0\cap I_1=\emptyset$.
    Cone $K$ is closed, convex  and has no interior points (see \cite{Fullerton_geometric_properties}).
     In the sequel we refer to cone $K$ as a basis cone (see \cite{MR147888}). 
    
    Let  $G:\ E\rightarrow F$  be    $C^1$ mappings, i.e. $G:\ E\rightarrow F$  are continuously Fr\'echet differentiable functions. Consider the minimization problem
\begin{equation}
\label{problem:P} 
    \tag{$P$}
    \operatorname{Minimize}_{x\in\mathcal{ F} }f_0(x)
    \end{equation}
    where $\mathcal{ F}:=\{x\in E\ \mid\ G(x)\in K\}$.
    
    By taking $F=\ell_2$,  $b_i=e_i$, $i\in \mathbb{N}$, where $(e_i)_{i\in \mathbb{N}}$, is the canonical basis of $\ell_2$,  
    the problem \eqref{problem:P} is of the form \eqref{problem:P0}.     Indeed, the set $ \mathcal{ F}$ takes the form
$$
    \mathcal{ F}:=\{x\in E\ \mid\ g_i(x):=e_i^*(G(x))=0,\ i\in I_0,\ \ g_i(x):=e_i^*(G(x))\leq 0,\ i\in I_1\}.
    $$
    

    In the case when $F$ is finite-dimensional  ($F = \mathbb{R}^\kappa$) and $b_i=e_i$, $i\in \{1,\dots,\kappa\} $ is the canonical basis of the space $F$, cone $K$ has a form
    \begin{equation}\label{cone:K_finite}
        K= \{ y=(y_1,\dots,y_\kappa) \in F \mid y_i= 0,\ i\in I_0,\ y_i\leq 0, \ i\in I_1 \},
    \end{equation}
where $I_0\cup I_1 = \{1,\dots,\kappa \}$. 

    The sufficient conditions for problem \eqref{problem:P} with $F$ - finite dimensional and cone $K$ given by \eqref{cone:K_finite} are considered in  \cite{MR4223867}. Below we concentrate on the case when $F$ is infinite-dimensional.

    \begin{remark}
    In Banach spaces $X$ with a basis $(b_i)_{i\in \mathbb{N}}$, the cone defined as
        \begin{equation}\label{cone:kbi}
        K_{\{b_i\}}:= \{ y \in X \mid y= \sum_{i\in \mathbb{N}} b_i^*(y) b_i, \ b_i^*(y)\leq 0, \ i\in \mathbb{N} \}
    \end{equation}
    was considered  in \cite{Fullerton_geometric_properties,MR313766}. For $I_0=\emptyset$ and $I_1=\mathbb{N}$, the cone $K_{\{b_i\}}$ coincides with $K$.  In the space $c$ cone defined as \eqref{cone:kbi} 
    can never coincide with the natural nonnegative cone, since the natural nonnegative cone has a nonvoid interior while the basis cone $K_{\{b_i\}}$  
    always has a void interior 
    (see \cite{MR147888}).
    \end{remark}
    
    Let us note that even in Hilbert spaces not all standard nonnegative cones  can be represented in the form of \eqref{cone:K} as can be seen from the following example\footnote{Example provided by professor Sergei Konyagin, Academician of RAS, by courtesy of profesor Nikolai Osmolovskii.}.
    
    \begin{example}
    \label{example:konyagin}
    Let $F=L_2([0,1])$. Suppose that $f_k \geq 0$, $k=1,\dots$ is a basis of the nonnegative cone of $L_2([0,1])$. Observe that, due to the fact that $L_2^+([0,1])$ is generating, $f_k$, $k\in \mathbb{N}$, is also a basis of $F$. For each $k\in \mathbb{N}$ let $F_k$ be the support of $f_k$ and choose a subset $M_k \subset F_k$ of measure $<10^{-k}$, Then the set $M=\bigcup_{k\in \mathbb{N}} M_{k}$ has measure $<1/9$, hence the set $M^\prime=[0,1]\setminus M$ has measure greater than $0$. 
    
    Let $h(t)$ be the characteristic function of $M^\prime$, i.e. $h(t)=1$ for $t\in M^\prime$ and $h(t)=0$ for $t\notin M^\prime$. Obviously, $h(t)=0$ on $M$. On the other hand,  $h(t)=\sum_{k\in \mathbb{N}} a_k f_k(t)$ with $a_k\geq 0$, $k\in \mathbb{N}$. Take any $k\in \mathbb{N}$ such that $a_k>0$. Then $h(t)>0$ on $F_k$ and in particular on $M_k$. However, $M_k\subset M$, where $h(t)=0$, a contradiction, i.e., such a basis does not exists. 
    \end{example}

We have
\begin{equation*}
    G(x)=\sum_{i\in I_0\cup I_1} b_i^*(G(x)) b_i=\sum_{i\in I_0\cup I_1} g_i(x) b_i, 
      \quad \text{where } g_i(x):= b_i^* (G(x)), \ i\in I_0\cup I_1.
\end{equation*}

By \eqref{cone:K},  we can rewrite set $\mathcal{ F}$ in an equivalent way as follows
\begin{equation}\label{set:constraints}
 \mathcal{ F}= \{ x \in E \mid g_i(x)=0,\ i \in I_0,\ g_i(x)\leq 0,\ i\in I_1 \}.
\end{equation}



  The aim of the present investigation  is to provide conditions ensuring the Abadie CQ,
    $$
    \mathcal{ T}_{\mathcal{ F}}(x_{0})=\Gamma_\mathcal{ F}(x_{0}),
    $$
    where $\mathcal{ T}_{\mathcal{ F}}(x_{0})$ is the tangent cone to $\mathcal{ F}$ at $x_0$ (see \eqref{eq_1}) and $\Gamma_{\mathcal{ F}}(x_{0})$ is the linearized cone to $\mathcal{ F}$ at $x_0$ (see \eqref{eq_2}). In the sequel we concentrate on the case $I_0\cup I_1=\mathbb{N}$.


Our main tool is Relaxed Constant Rank Constraint Qualification Plus 
introduced in Section \ref{section:Abadie_condition}. Other regularity conditions were recently proposed in  \cite{MR3579742,constant_rank_constraint_Andreani,MR4223867,MR4159570,MR3070104}.

The novelty of our approach relies on the  use of the Schauder bases and basic sequences in definitions of the main concepts, namely in the definitions of the CRC+ (Definition \ref{definition:crcplus}) and RCRCQ+ (Definition \ref{rcrcq_condition}).
Consequently, the assumptions of the main theorem (Theorem \ref{theorem:tangent_cone}), of Proposition \ref{proposition:functional_dependence} on functional dependence and two auxiliary results (Proposition  \ref{proposition:E1_representation}, Proposition \ref{proposition:CRC_and_isomorphism}) are expressed with the help of Schauder bases of some spaces generated by the derivative of the constraint map.

Moreover, in the proof of Proposition  \ref{proposition:functional_dependence} is based on Rank Theorem \ref{theorem:rank}. Up to our knowledge, the only result concerning the existence  Lagrange multipliers for which the proof is based on Rank Theorem is  \cite[Theorem 4.1]{MR4104521}.

\begin{theorem}(\cite[Theorem 4.1]{MR4104521})\label{theorem:blot} Let $x_0$ be a local solution to problem \eqref{problem:P} with $K=\{0\}$ and $F$ be a Banach space. Let $x_0$ be a local solution to \eqref{problem:P}. Assume that the following conditions are fulfilled:
\begin{enumerate}[label=\emph{(\Alph*)}]
    \item\label{assumption_blot1} $f_0$ is Fr\'echet differentiable at $x_0$ and $G$ is of class $C^1$ in a neighbourhood of $x_0$.
    \item\label{assumption_blot2} $E_2:= \ker DG(x_0)$ is topologically complemented in $E$, i.e., $E=E_1 \oplus E_2$, where $E_1$ is a closed subspace of $E$, $F_1:= DG(x_0)(E)$ is closed and topologically complemented in $F$,
    i.e., $F=F_1 \oplus F_2$, where $F_2$ is a closed subspace of $F$.
    \item\label{assumption_blot3} There exists a neighbourhood of $x_0$ such that for all $x$ in this neighbourhood $DG(x)(E)\cap F_2= \{0\}$.
\end{enumerate}
Then there exists $\lambda\in F^*$ such that $Df_0(x_0)=\lambda \circ DG(x_0)$.
\end{theorem}

\begin{remark}\label{remark:assumptionC}
By Proposition 1 of \cite{MR857736}, condition \ref{assumption_blot3} is equivalent to condition 
\begin{equation*}
DG(x)|_{E_1} :\ E_1 \rightarrow DG(x)(E)\quad \text{is an isomorphism.}
\end{equation*}
\end{remark}

In Proposition \ref{proposition:E1_representation} of Section \ref{section:Complemented_and_isomorphism} we provide conditions under which we ensure that \ref{assumption_blot2} holds. 
In Proposition \ref{proposition:CRC_and_isomorphism} of Section \ref{section:Complemented_and_isomorphism} we provide conditions under which we ensure that \ref{assumption_blot3} holds. 

The organization of the paper is as follows. In Section \ref{section:preliminaries} we recall basic concepts related to Schauder bases.
 In Section \ref{section:CRC} we introduce Constant Rank Condition (CRC) and Constant Rank Condition Plus (CRC+). In Section \ref{section:Complemented_and_isomorphism}, Conditions \ref{assumption:Aplus}, \ref{assumption:Bplus}, \ref{assumption:Cplus} of CRC+ allows us to prove Proposition \ref{proposition:E1_representation} on the split of $E$. Condition CRC+ allows us to prove Proposition \ref{proposition:CRC_and_isomorphism} on isomorphisms. Section \ref{section:CRCandfunctional} is devoted to the proof of Proposition \ref{proposition:functional_dependence} on functional dependence. In Section \ref{section:Abadie_condition} we introduce Relaxed Constant Rank Constraint Qualification Plus, which is the main ingredient of main results of Section \ref{section:main_result}. In Section \ref{section:main_result} we present main result, namely sufficient conditions under which Abadie condition holds for the investigated problem. The closing Section \ref{section:lagrange_multipliers} is devoted to the topic of existence of Lagrange multipliers to problem \eqref{problem:P}.




\section{Preliminaries}\label{section:preliminaries}

 \begin{definition}
Let $X$ be a normed linear space. A sequence $(b_{i})_{i\in\mathbb{N}}$
in $X$ is called a Schauder basis of $X$ (or simply basis) if for every $x\in X$ there is a unique sequence of scalars $(a_i)_{i\in \mathbb{N}}$, called the coefficients of $x$, such that $x=\sum_{i\in \mathbb{N}} a_i b_i$.
 \end{definition}
 
 If $X$ is finite-dimensional, the notion of Schauder basis coincides with the linear  basis (see \cite[Example (i), Chapter 6]{MR1831176}).

We say that sequence $(b_i)_{i\in \mathbb{N}}\in F$ is a \textit{basic sequence} if $(b_i)_{i\in \mathbb{N}}$ is Schauder basis of $\spanc ( b_i, i\in \mathbb{N} )$,  where $\spanc$ denotes the closure of the span and the closure is taken in the strong topology of the space.

\begin{definition}
For a given, possibly nonconvex, set $Q\subset E$ and $x\in Q$  the {\em tangent (Bouligand) cone} to $Q$ is  defined as
\begin{equation} 
\label{eq_1} 
\mathcal{ T}_{Q}(x):=\{d\in E\ |\ \exists\ \{x_{k}\}\subset Q\,\ \{t_{k}\}\subset \mathbb{R}\,\ x_{k}\rightarrow x,\, t_{k}\downarrow 0,\ (x_{k}-x)/t_{k}\rightarrow d\}.
\end{equation}


The cone
\begin{equation} 
\label{eq_2}
\Gamma_\mathcal{ F}(x):=\{d\in E\ \mid DG(x)d\in \mathcal{ T}_{K}(G(x))\}
\end{equation}
is called the {\em linearized cone to $\mathcal{ F}$ } at $x$.
\end{definition}

\begin{definition}
    For any $x\in \mathcal{ F}$, where $\mathcal{ F}$ is given by \eqref{set:constraints},   $I(x)$ denotes the set of  active  (inequality) indices of $\mathcal{ F}$ at $x$, 
    \begin{equation*}
        I(x):=\{ i\in I_1 \mid g_i(x)=0 \}.
    \end{equation*}
 \end{definition}

\begin{example} 
	Let $G:E\rightarrow \mathbb{R}^{n}$, $G(x)=(g_{1}(x),\dots,g_{n}(x))$ and $K=\mathbb{R}_{-}^n=\{y=(y_{i})\in\mathbb{R}^{n}\ |\ y_{i}\le 0,\ i=1,\dots n\}$, i.e.
	$$
	\mathcal{ F}=\{x\in E\ | \ g_{i}(x)\le 0, \ i=1,\dots, n\}
	$$
	where $g_{i}:E\rightarrow\mathbb{R}$, $i=1,\dots, n$. 
	Let us calculate  $\mathcal{ T}_{K}(G(x_0))$, $x_0\in E$,  $G(x_0)\in\text{cl\,} K=K$.
	\begin{enumerate}
	    \item If $G(x_0)\in \text{int\,} \mathbb{R}_{-}^n$, then $\mathcal{ T}_{K}(G(x_0))=\mathbb{R}^n.$
	    \item If $G(x_0)\in \text{bd\,} \mathbb{R}_{-}^n$, then  $I(x_0):= \{i\in \{1,\dots,n\}\ |\ g_{i}(x_0)=0\}$ is nonempty and $\mathcal{ T}_{K}(G(x_0))=\mathbb{R}_{-}^n=\{y=(y_{i})\in\mathbb{R}^{n}\ |\ y_{i}\le 0,\ \ i\in I(x_0)\}$.
	\end{enumerate}
	
	\end{example}

\begin{proposition}\label{proposition:representation_tangent}
For any $y_0\in K$,  where $K$ is given by \eqref{cone:K}, we have
\begin{equation}
\label{nonnegative_cone_formula}
\mathcal{ T}_{K}(y_{0}) = \{z\in F\ \mid\ b_i^*(z) \leq 0,\ i\in I(y_{0}), \ b_i^*(z)=0,\ i\in I_0\},
\end{equation}
where $I(y_{0}):=\{i\in I_1\ \mid\ b_i^*(y_0)= 0\}$.
\end{proposition}
\begin{proof}

 Take any $\tilde{z}\in \mathcal{ T}_{K}(y_{0})$. By the definition, there exist $r(t)\in F$ and $\varepsilon_0>0$ such that $r(t)/t\rightarrow 0^+$ and
 \begin{align} 
 \label{eq_def1} 
 y_{0}+t\tilde{z}+r(t)\in K \ \ \  \forall\ \ t\in[0,\varepsilon_{0}).
 \end{align}
 By this, for all $t\in[0,\varepsilon_{0})$
 \begin{align} 
 \label{eq_def2} 
 \begin{aligned}
 & b_{i}^* (t\tilde{z})+ b_{i}^* ( r(t))\le 0\ \ \ \forall\ \ i\in I(y_{0}),\\
 &b_{i}^* ( t\tilde{z})+b_{i}^*( r(t))= 0\ \ \ \forall\ \ i\in I_0,\end{aligned}
 \end{align}
 and consequently
\begin{align*}
 &  b_{i}^*  ( \tilde{z}\rangle+  b_{i}^* ( \frac{r(t)}{t})\le 0\ \ \ \forall\ \ i\in I(y_{0}),\\
  &  b_{i}^*  ( \tilde{z}\rangle+  b_{i}^* ( \frac{r(t)}{t})= 0\ \ \ \forall\ \ i\in I_0.
 \end{align*}
 Since $r(t)/t\rightarrow 0$ as $t\rightarrow 0^+$, we obtain 
$\mathcal{ T}_{K}(y_{0})\subset\{z\in F\ \mid\   b_{i}^* ( z)\le 0,\ i\in I(y_{0}),\   b_{i}^* ( z)=0,\ i\in I_0\} $. 

Now, to see the converse, take any $\tilde{z}\in \{z\in F\ \mid\   b_{i}^* (\tilde{z})\le 0,\ i\in I(y_{0}),\   b_{i}^* (\tilde{z})=0,\ i\in I_0\}$. We have
\begin{equation*}
    \tilde{z}=\lim_{k\rightarrow +\infty }z_k = \sum_{i=1}^k b_i^*(z)b_i.
\end{equation*}
Since $z_k\in \mathcal{ T}_{K}(y_{0})$, $k\in \mathbb{N}$ and $\mathcal{ T}_{K}(y_{0})$ is closed, $\tilde{z}\in \mathcal{ T}_{K}(y_{0})$.

\end{proof}

\begin{fact}\label{fact:representation_bstar}
Let $\bar{x} \in E$.  Since $G$ is assumed differentiable on $E$ we have
\begin{align*}
0=   &\lim_{h\rightarrow 0} \frac{G(\bar{x}+h)-G(\bar{x})-DG(\bar{x})h}{\|h\|} \\
    &=\lim_{h\rightarrow 0}\frac{ \sum_{i\in \mathbb{N}} b_i^*(G(\bar{x}+h)-G(\bar{x})-DG(\bar{x})h)b_i}{\|h\|}\\
    &=\sum_{i\in \mathbb{N}} \lim_{h\rightarrow 0}\frac{  b_i^*(G(\bar{x}+h)-G(\bar{x})-DG(\bar{x})h)b_i}{\|h\|}\\
    & = \sum_{i\in \mathbb{N}} \lim_{h\rightarrow 0}\frac{ g_i(\bar{x}+h)-g_i(\bar{x})-Dg_i(\bar{x})h}{\|h\|}  b_{i}.
\end{align*}
 Hence, due to the uniqueness of the representation of elements of the space $F$ in basis $(b_{i})_{i\in\mathbb{N}}$ the coefficients $g_{i}(\cdot)=b_i^{*}(g(\cdot))$, $i\in\mathbb{N}$ are differentiable for all $x\in E$
and $b_i^*( DG(\bar{x})z)= Dg_i(x)z $ for any $z\in E$, $i\in \mathbb{N}$. 
\end{fact}

As a consequence we obtain the following proposition.

\begin{proposition} \label{proposition:representation_linearized}
Let $x_0\in \mathcal{ F}$, where $\mathcal{ F}$ is given by \eqref{set:constraints}. Then
\begin{align} 
\label{linearized_cone_formula}
\begin{aligned}
\Gamma_{\mathcal{ F}}(x_{0})=&\{d\in E \mid   DG(x_{0}) d\in \mathcal{ T}_{K}(G(x_{0}))\}\\
\subset &\{d\in E \mid   Dg_{i}(x_{0}) d= 0,\ \ i\in I_0,\\ &  Dg_{i}(x_{0}) d\le 0,\ \ i\in I(x_{0})\},
\end{aligned}
\end{align}
where $I(x_{0}):=\{i\in I_1 \mid g_{i}(x_{0})=0\}$. 
\end{proposition}
\begin{proof}

Take any $d \in \Gamma_\mathcal{ F}(x_0)$. This means that 
\begin{equation}\label{inclusion:intangent}
      DG(x_0)  d  \in \mathcal{ T}_{K}(G(x_0)).
    \end{equation}
By Proposition \ref{proposition:representation_tangent},
\begin{equation*}
    \mathcal{ T}_{K}(G(x_0))\subset \{z\in F\ \mid\ b_i^*(z) = 0,\ i\in I_0,\ b_i^*(z) \leq 0,\ i\in I(G(x_0))\}.
\end{equation*}
Hence, by \eqref{inclusion:intangent},
\begin{equation*}
    b_i^*(  DG(x_0)  d )= 0, \ i \in I_0 ,\ b_i^*(  DG(x_0)  d )\leq 0, \ i \in I(G(x_0)).
\end{equation*}
Now \eqref{linearized_cone_formula} follows from Fact \ref{fact:representation_bstar}. 
 

\end{proof}

  \subsection{ Boundedly-complete, shrinking, Besselian and Hilbertian bases}
 
In this subsection we recall basic definitions and facts related selected types of bases in Banach spaces. These concepts  will be extensively used  in the sequel.
 \begin{definition}
The closed subspace $E_1$ of the Banach space $E$ is said to be split, or complemented,
if there is a closed subspace $E_2 \subset E$ such that $E = E_1 \oplus E_2$.
\end{definition}

\begin{proposition}\label{proposition_splits} (\cite[Theorem 2.1.15]{manifolds_tensor_vol2})
	If $F$ is a Hilbert space and $F_1$ a closed subspace, then $F=F_1\oplus F_1^\perp$. Thus every closed subspace of a Hilbert space splits (see e.g. Definition 2.1.14 of \cite{manifolds_tensor_vol2}).
\end{proposition}

 \begin{definition}(\cite[Definition 3.2.8]{MR2192298})
 Let $X$ be a Banach space. A sequence $(v_n)_{n\in \mathbb{N}}$ in $X$ is \textit{boundedly-complete} if whenever $(a_n)_{n\in \mathbb{N}}$ is a sequence of scalars such that
 \begin{equation*}
     \sup_{N\in \mathbb{N}} \| \sum_{n=1}^N a_n v_n \|< +\infty
 \end{equation*}
 then the series $\sum_{n\in \mathbb{N}}a_n v_n$ converges.
 \end{definition}
 
 \begin{remark}\label{remark:boundedlycompletesubseq}
 Let $(v_n)_{n\in \mathbb{N}}$ in $X$  be boundedly-complete.  Then  every subsequence $(v_{n_k})_{k \in \mathbb{N}} $ is boundedly-complete. Indeed,  suppose that $(v_n)_{n\in \mathbb{N}}$ is boundedly-complete and   $(a_{k})$, $k\in \mathbb{N}$, is such that $\sup_{N\in \mathbb{N}} \|\sum_{k=1}^N a_{k} v_{n_k} \|< \infty$. Let
 \begin{equation*}
     \bar{a}_k=\left\{ \begin{array}{ll}
          a_{k},   & k\in \{n_l\}_{l\in \mathbb{N}},\\
          0, & \text{otherwise}.
     \end{array} \right.
 \end{equation*}
 Then
 \begin{align*}
     &\sup_{N\in \mathbb{N}} \|\sum_{k=1}^N a_{k} v_{n_k} \|< \infty \iff \sup_{N\in \mathbb{N}} \|\sum_{n=1}^N \bar{a}_n v_n \|< \infty \\
     & \implies \sum_{n\in \mathbb{N}} \bar{a}_n v_n \text{ converges} \iff \sum_{k\in \mathbb{N}}a_{k} v_{n_k} \text{ converges}.
 \end{align*}

 \end{remark}
 \begin{definition}(\cite[Definition 3.2.5]{MR2192298})
 A basis $(v_n)_{n\in \mathbb{N}}$ of a Banach space $X$ is \textit{shirnking} if the sequence of its bioorthogonal functionals $(v_n^*)_{n\in \mathbb{N}}$ is a basis of $X^*$, i.e., $\spanc(v_n^*,\ n\in \mathbb{N})=X^*$.
 
 \end{definition}
 
 \begin{proposition}(\cite[Theorem 3.2.10]{MR2192298})\label{propozycja-izom}
Let $(v_n, v_n^*)$ be a biorthogonal system in    
a Banach space $X$. The following are equivalent:
\begin{enumerate}
    \item $(v_n)_{n\in \mathbb{N}} $ is boundedly-complete,
    \item $(v_n^*)_{n\in \mathbb{N}} $ is shrinking basis for $\spanc(v_n^{*},\ n\in \mathbb{N})$ 
    \item the canonical map $\operatorname{eval}_X:\ X\rightarrow \spanc(v_n^*,\ n\in \mathbb{N})^*$ defined by $\operatorname{eval}_X (x)(h)=h(x)$ for all $x\in X$ and $h\in \spanc(v_n^*,\ n\in \mathbb{N})$, is an isomorphism.
\end{enumerate}
\end{proposition}

\begin{remark} (Corollary 3.2.11 of \cite{MR2192298}) 
Every Schauder basis of a Hilbert space is boundedly-complete. $c_0$ has no boundedly-complete basis.
\end{remark}

Let us recall James Theorem from 1951.
\begin{theorem}\label{theorem:James}
(James theorem, see \cite[Theorem 3.2.13]{MR2192298})
Let $X$ be  a Banach space. If $X$ has a basis $(v_n)_{n\in \mathbb{N}}$ then $X$ is reflexive if and only if $(v_n)_{n\in \mathbb{N}}$ is both  boundedly-complete and shrinking. 
\end{theorem}

\begin{definition}(see Definition 11.1 of \cite{MR0298399}) 
We say that basis $(c_i)_{i\in \mathbb{N}}$ in a real Banach space is
\begin{enumerate}
    \item \textit{Besselian} if
\begin{equation*}
    \sum_{i\in \mathbb{N}} \alpha_i c_i \quad \text{converges} \implies \sum_{i\in \mathbb{N}} (\alpha_i)^2<+\infty .
\end{equation*}
    \item \textit{Hilbertian} if
    \begin{equation*}
    \sum_{i\in \mathbb{N}} (\alpha_i)^2<+\infty  \implies \sum_{i\in \mathbb{N}} \alpha_i c_i \quad \text{converges},
\end{equation*}
i.e., for every $\alpha_i \in \mathbb{R}$, $i\in \mathbb{N}$,  with $\sum_{i\in \mathbb{N}} (\alpha_i)^2<+\infty $ there exists an (obviously unique) $x$ such that
\begin{equation*}
    c_i^*(x)=\alpha_i, \quad i\in \mathbb{N}.
\end{equation*}
\end{enumerate}

\end{definition}
\begin{remark}
The natural basis of $\ell_2$ is Besselian. 
Not all bases in Hilbert spaces are Besselian or Hilbertian. For $L_2[-\pi,\pi]$ see  \cite[Example 11.2]{MR0298399}. 
\end{remark}

\begin{remark}(Corollary 4.6 of \cite{MR331025})
The space $L_1[0,1]$ has a Besselian basis. 
\end{remark}

\begin{lemma} (\cite[Lemma 3.1]{MR303262}) \label{lemma:besseliantohilbertian}
 The basis $(c_i)_{i\in \mathbb{N}}$ of Banach space $X$ is Hilbertian (Besselian) if and only if the basic sequence $(c_i^*)_{i\in \mathbb{N}}$ in $X^*$ is Besselian (Hilbertian). 
\end{lemma}

  \begin{proposition}(\cite[Proposition 3.1]{MR0298399})\label{proposition:Y_Banach} Let $\{v_n\}$ be a sequence in Banach space $X$, and assume that $v_n\neq 0$ for every $n$.
  
  Define $Y:=\{ (c_n)\mid \sum_{n} c_n v_n\quad \text{converges in } X\} $ and set
 \begin{equation}\label{normY}
     \|(c_n)\|_Y:= \sup_{N} \| \sum_{n=1}^N c_n v_n \|.
 \end{equation}
 Then the following hold:
 \begin{enumerate}
     \item $Y$ is a Banach space.
     \item If $\{v_n\}$ is a basis for $X$, then $Y$ is topologically isomorphic to $X$ via the mapping $(c_n)\mapsto \sum_{n} c_n v_n$.
 \end{enumerate}
 \end{proposition}
 
  For any infinite subset $J=\{ j_1,j_2,\dots \}\subset \mathbb{N}$ let us denote
 \begin{equation}\label{def:Y(J)}
     Y(J)=\{ (c_i)_{i\in J}\mid \sum_{i \in J} c_i v_i\quad \text{converges in } X\}.
 \end{equation}
 with associated norm 
 \begin{equation}\label{normY2}
         \|(c_i)\|_{Y(J)}:= \sup_{N} \| \sum_{i=1}^N c_{j_i} v_i \|
 \end{equation}
and $\ell_2(J):=\{ (c_j)_{j\in J} \mid \sum_{i\in J} (c_j)^2 < +\infty  \}$ with associated norm
\begin{equation*}
    \| (c_j)_{j\in J} \|_{\ell_2(J)}= \sum_{i\in J} (c_j)^2.
\end{equation*}

 
 \begin{remark}
 \label{coefficient:Hilbert}
 Observe that if $X$ is a Hilbert space with the inner product $\langle \cdot\ |\ \cdot\rangle$, the norm $\|x\|=\sqrt{\langle x\ |\ x\rangle}$ and  the orthonormal basis $(x_{i})_{i\in J}$, then
 $$
 \|(c_{i})\|_{Y(J)}\begin{array}[t]{l}
 =\sup_{N}\sqrt{\langle \sum_{i=1}^N c_{j_i} x_i\ |\ \sum_{i=1}^N c_{j_i} x_i\rangle}\\
 =\sup_{N}\sqrt{\sum_{i=1}^N\sum_{k=1}^N\langle  c_{j_i} x_i\ |\ c_{j_k} x_k\rangle}\\
 =\sup_{N}\sqrt{\sum_{i=1}^N (c_{j_i})^{2}}=\sqrt{\sup_{N}\sum_{i=1}^N (c_{j_i})^{2}} \\
 =\sqrt{\sum_{i\in J} (c_{i})^{2}}=\|(c_{i})_{i\in J}\|_{\ell_{2}(J)}. 
 \end{array}
 $$
 By James Theorem (Theorem \ref{theorem:James}), $(x_{i})_{i\in J}$  is boundedly-complete in $X$. 
 In view of this, when $(x_{i})_{i\in J}$ is an orthonormal
 , then 
 $Y(J)=\ell_{2}(J):=\spanc(e_{i},\ i\in J)$, where $e_{i}^{j}=0$ if $j=i$ and $0$ otherwise, for all $i,j\in J$.
 \end{remark}
 
 \section{Constant rank condition}\label{section:CRC}
 
 In this section we introduce the Constant Rank Condition (CRC)\footnote{Let us note that the same terminology (Constant Rank Condition) in finite-dimensional case has been already used in  \cite{MR2679662} (Definition 1) and \cite{directional_derivative_Janin}, and differs from that proposed in  Definition \ref{definition:crc}.} for a possibly infinite family of functions defined on a Banach space via Schauder basis. For other forms of CRC which do not refer to Schauder basis see e.g. \cite{MR4104521}.

Let $E$ be a Banach space and $F$ a Hilbert space with basis $(b_{i})_{i\in\mathbb{N}}$. Consider $f:E\rightarrow\spanc(b_{i},\ i\in \mathbb{N})$, i.e. $f(x)=\sum_{i\in \mathbb{N}}f_{i}(x)b_{i}$, where $f_{i}=b_{i}^{*}(f(x)):E\rightarrow\mathbb{R}$, $i\in \mathbb{N}$ are continuous functionals.

  \begin{definition}
  \label{definition:crc}
Let  $(f_{i})_{i\in J_1}:\ E \rightarrow Y(J_1)$, $J_1\subset \mathbb{N}$ be of class $C^1$. We say that  the Constant Rank Condition (CRC in short) holds for  $(f_{i})_{i\in J_1}$ at $x_0\in E$ 
if there exist
a neighbourhood $V(x_0)$ and a subset $J_2\subset J_1$ such that

\begin{enumerate}[label=\emph{\arabic*.}]
\item \label{item:CRC:1} $(Df_i(x))_{i\in J_2}$ 
forms a Schauder basis for 
 $\spanc (D f_{i}(x), i \in J_1)$ for all $x\in V(x_0)$,

 \item \label{item:CRC:3} for any $x\in V(x_0)$, there exists a topological isomorphism (linear)
	\begin{equation}\label{isomorhpism_zx}
	    z_x: \spanc ( D f_{i}(x), i \in J_2) \rightarrow \spanc (D f_{i}(x_0), i \in J_2), 
	\end{equation}
	such that $z_x(D f_{i}(x))= D f_{i}(x_0)$, $i \in J_2$.
	\end{enumerate}
	
\end{definition}

We interpret $z_x$ as an isomorphism of functionals, i.e., $z_x(Df_i(x)(\cdot))=Df_i(x_0)(\cdot)$, $x\in V(x_0)$, $i\in J_2$, hence, by \eqref{isomorhpism_zx}, $(Df_i(x_0))_{i\in J_2}(E)$ and $(Df_i(x))_{i\in J_2}(E)$ are isomorphic for all $x\in V(x_{0})$.

 \begin{remark}
Let us note that, for $(f_{i})_{i\in J_1}$, $J_1\subset \mathbb{N}$, where $J_1$ is finite, 
the condition
\begin{equation*}
    \rank \{ Df_i(x_0), i\in J_1 \} = \rank \{ Df_i(x), i\in J_1 \}, \quad x\in V(x_0),
\end{equation*}
where $V(x_0)$ is a neighbourhood of $x_0$, 
is equivalent to the existence of isomorphism $z_x$ given in Definition \ref{definition:crc} for $(f_{i})_{i\in J_1}$ at $x_0$ (see \cite[Definition 2.1]{MR4223867}).

\end{remark}

\begin{definition}(\cite[Definition 1.3.1]{MR2192298})
Two bases (or basic sequences) $(u_n)_{n\in \mathbb{N}}$ and $(v_n)_{n\in \mathbb{N}}$ in the respective real Banach spaces $X$ and $Y$ are said to be \textit{equivalent}, if whenever we take a sequence of scalars $(a_n)_{n\in \mathbb{N}}$, then $\sum_{n\in \mathbb{N}}a_n u_n$ converges if and only if $\sum_{n\in \mathbb{N}}a_n v_n$ converges.
\end{definition}

 Let us recall the following Fact.
\begin{fact} (\cite[Theorem 1.3.2]{MR2192298})
Let $(u_i)_{i\in J}$, $J\subset \mathbb{N}$, be a basic sequence in a Banach space $X$ and let $(v_i)_{i\in J}$ be a basic sequence in a Banach space $Y$. The following are equivalent:
\begin{itemize}
    \item $(u_i)_{i\in J}$ is a basic sequence equivalent to $(v_i)_{i\in J},$
    \item There is an isomorphism $T$ of $\spanc (u_i )_{i\in J}$ onto $\spanc (v_i)_{i\in J}$ s.t. $T(u_i)=v_i$ for every $i\in J$.
\end{itemize}
\end{fact}
In other words, condition \eqref{item:CRC:3} of CRC means that a basic sequence $(Dg_i(x))_{i\in J_2}$ is equivalent to a basic sequence $(Dg_i(x_0))_{i\in J_2}$  in $E^{*}$ for all $x$ from some neighbourhood $V(x_0)$.

In the sequel we will use CRC with some additional conditions, e.g. in Definition \eqref{rcrcq_condition}, which motivates the following definition. 

\begin{definition}
  \label{definition:crcplus}
Let  $(f_{i})_{i\in J_1}:\ E \rightarrow Y(J_1)$, $J_1\subset \mathbb{N}$ be of class $C^1$. We say that  the Constant Rank Condition Plus (CRC+ in short) holds for  $(f_{i})_{i\in J_1}$ at $x_0\in E$ 
if there exist
a neighbourhood $V(x_0)$ and a subset $J_2\subset J_1$ such that

\begin{enumerate}[label=\emph{\arabic*.}]
\item \label{item:CRC:1plus} $(Df_i(x))_{i\in J_2}$ 
forms a Schauder basis for 
 $\spanc (D f_{i}(x), i \in J_1)$ for  $x\in V(x_0)$,
  \item \label{item:CRC:3plus} for all $x\in V(x_0)$, there exists a topological isomorphism
  	\begin{equation*}
	    z_x: \spanc ( D f_{i}(x), i \in J_2) \rightarrow \spanc (D f_{i}(x_0), i \in J_2), \mbox{ for all } x\in V(x_0)
	\end{equation*}
	such that $z_x(D f_{i}(x))= D f_{i}(x_0)$, $i \in J_2$
\end{enumerate}
and additionally
\begin{enumerate}[label=\emph{\arabic*.}]
\setcounter{enumi}{2}
 \item \label{item:CRC:2plus} $(Df_i(x_0))_{i\in J_2}(E)$ is closed in $Y(J_2)$ defined in \eqref{def:Y(J)} with $v_i=b_i$, $i\in J_2$,
    \item\label{assumption:Aplus}
 $(Df_i(x_0))_{i\in J_2}$ forms a basis which is shrinking and boundedly-complete for $\spanc(Df_i(x_0),\ i\in J_1)$, equivalently  $\spanc(Df_i(x_0),\ i\in J_1)$ is reflexive and $(Df_i(x_0))_{i\in J_2}$ forms a basis of this space,
    \item \label{assumption:Bplus} $(Df_i(x_0))_{i\in J_2}$ is Besselian for $\spanc(Df_i(x_0),\ i\in J_1)$ 
    \item \label{assumption:Cplus} $Df_i(x_0)^*\in E$, $i\in J_2$.
	\end{enumerate}
	
\end{definition}

In view of Remark \ref{coefficient:Hilbert}, if $(b_{i})_{i\in J_{2}}$ is orthonormal, then, in Definition \ref{definition:crcplus}, \ref{item:CRC:2plus} we have $Y(J_{2})=\ell_{2}(J_{2})$.

\begin{remark}
Let us remark that, if $J_1$ is finite then \ref{item:CRC:2plus},  \ref{assumption:Aplus}, \ref{assumption:Bplus}, \ref{assumption:Cplus} are automatically satisfied.
\end{remark}
 
 
 In the definition of CRC+ we assume closedness of $((Df_i(x_0))_{i\in J_2})(E)$. Let us recall that such sets do not have to be closed as shown in the example below.
\begin{example} (\cite{MR771117})
Let $G:\ \ell_2 \rightarrow \ell_2$ be defined as $G(x)=(\frac{1}{2i}x_i^2)_{i\in \mathbb{N}}$. Then
$DG(\cdot ):=\sum_{n=1}^{\infty} \frac{1}{n} e_n \langle e_n \mid \cdot \rangle$ and
\begin{enumerate}
    \item $DG(x)\in L(\ell_2,\ell_2)$ (with $\|DG(x)\|=1$)
    \item $v\in  DG(x)(\ell_2)$ if and only if $\sum_{n\in \mathbb{N}} n^2 |\langle e_n \mid v \rangle|^2<+\infty$
    \item With $v_0:=\sum_{n\in \mathbb{N}} \frac{1}{n^{\frac{3}{2}}}e_n$ and $\{v_j\}:=\sum_{n\in \mathbb{N}} \frac{1}{n^{\frac{3}{2}+\frac{1}{j}}}e_n$ we have $v_0\notin DG(x)(\ell_2)$ yet $v_j\in  DG(x)(\ell_2)$ and $v_j\rightarrow v_0$ as $j\rightarrow +\infty$.

\end{enumerate}
     
\end{example}

\section{Complemented kernels and isomorphisms}\label{section:Complemented_and_isomorphism}

 In this section we prove  Proposition \ref{proposition:E1_representation} and Proposition \ref{proposition:CRC_and_isomorphism} which, together with CRC+ will be used in the next sections.

We start with Proposition \ref{proposition:E1_representation} providing conditions for a $C^{1}$ mapping $f$   ensuring that the kernel of its derivative $Df(x_{0})$ is complemented. It is next used in the proof of Lemma \ref{lemma:basis2} which is, in turn, used in the proof of Proposition \ref{proposition:CRC_and_isomorphism}.  Proposition \ref{proposition:E1_representation} could be of independent interest of itself.

\begin{proposition}\label{proposition:E1_representation}
Let $E$ be a Banach space, $F$ be Banach space with a Besselian Schauder basis $(b_i)_{i\in \mathbb{N}}$. Let  
$(f_{i})_{i\in J}:\ E \rightarrow Y(J)$, $J\subset \mathbb{N}$ be of class $C^1$, $f:\ E \rightarrow F$, $f(x):=\sum_{i\in J} f_i(x)b_i  $, $x\in E$.  
Let $x_0\in E$ and  $E_2=\ker Df(x_0)$, $X_1:=\spanc(Df_i(x_0),\ {i\in J})$, and assume that:
\begin{enumerate}[label=\emph{(\Alph*)}]
    \item\label{assumption:A}
 $(Df_i(x_0))_{i\in J}$ forms a shrinking and boundedly-complete basis for $X_1$, 
    \item \label{assumption:B} $(Df_i(x_0))_{i\in J}$ is Besselian for $X_1$,
    \item \label{assumption:C} $Df_i(x_0)^*\in E$, $i\in J$.
\end{enumerate}
Then $E=E_1 \oplus E_2$, where $E_1= \spanc (Df_i(x_0)^*,\ {i\in J}).$ Moreover, $X_1$ is a reflexive space.
\end{proposition}

\begin{proof} By \ref{assumption:A}, the fact that $X_1$ is a reflexive space follows immediately from James theorem (Theorem \ref{theorem:James}). Let  $v_{i}:=D f_i(x_0)\in E^*$, $i\in J$.
Since $(Df_i(x_0))_{i\in J}$ forms a boundedly-complete basis for $X_1:=\spanc (Df_i(x_0),\ {i\in J})$, by Proposition \ref{propozycja-izom},   there exists canonical isomorphism $\operatorname{eval}_{X_1}:\ X_1\rightarrow \spanc ( (Df_i(x_0)^*,\ {i\in J})^*$  defined as
\begin{equation*}
    \operatorname{eval}_{X_1}(v)(u^*)= u^* (v) \quad \text{for every} \   v\in X_1,\ u^*\in \spanc(v_n^*).
\end{equation*}
We have $v_{i}^{*}(v_{j})=v_{j}(v_{i}^{*})=Df_i(x_0)Df_{j}(x_{0})^{*}$ for all $i\in J$. 
By \eqref{delta_1}, \ref{assumption:C}, Fact \ref{fact:representation_bstar} and Proposition \ref{propozycja-izom}
\begin{align} 
\begin{aligned}\label{delta_2ss}
	&Df(x_0)(Df_i(x_0)^*)=
	\sum_{j\in J} Df_j(x_0)(Df_i(x_0)^*)b_j=b_i
	,\ i\in J.
	\end{aligned}
	\end{align}
	
		Now we show that $E=X_{1}^{*}\oplus \text{ker }Df(x_{0})$. For any  $x\in E$,  $Df(x_{0})(x)=
		\sum_{j\in J} Df_j(x_0)(x)b_j=\sum_{i\in J} \alpha_i(x) b_i$, $\alpha_i(x)=Df_i(x_0)(x)\in \mathbb{R}$. 
		Since $(b_i)_{i\in \mathbb{N}}$ is Besselian for $F$, $\sum_{i\in J} (\alpha_i(x))^2<+\infty$. 
	Let
	$m:=\sum_{i\in J}\alpha_i(x) Df_{i}(x_{0})^{*}$. By \ref{assumption:B} and Lemma \ref{lemma:besseliantohilbertian}, $(Df_i(x_0)^*)_{i\in J}$ is Hilbertian for $X_1^*$, $m$ is well defined, i.e., $m\in X_{1}^{*}$.

	By \eqref{delta_2ss}, we have
	\begin{align*}
	    	& Df(x_{0})(x-m)=Df(x_{0})(x)-Df(x_{0})(m)\\
	    	&= \sum_{i\in J}\alpha_{i}(x)b_{i}-  Df(x_0)(\alpha_i(x)\sum_{i\in J} Df_i(x_0)^*) \\
	    	& =\sum_{i\in J}\alpha_{i}(x)b_{i}-\sum_{i\in J}\alpha_i(x)Df(x_{0})(Df_{i}(x_{0})^{*})\\
	    	&=\sum_{i\in J}\alpha_{i}(x)b_{i}-\sum_{i\in J}\alpha_i(x)b_i=0.
	\end{align*}
	This shows that $x-m\in\text{ker\,} Df(x_{0})$  which proves the assertion with $E_{1}:=X_{1}^{*}$ and $E_{2}:=\text{ker\,} Df(x_{0}).$ 

\end{proof}

\begin{remark}
Let us note that space $\ell_1$  contains no infinite-dimensional reflexive subspaces.
\end{remark}

\begin{corollary}
 Let $E$ be a reflexive Banach space, $F$ be Hilbert space with a Besselian Schauder basis $(b_i)_{i\in \mathbb{N}}$. Let  
$(f_{i})_{i\in J}:\ E \rightarrow Y(J)$, $J\subset \mathbb{N}$ be of class $C^1$, $f:\ E \rightarrow F$, $f(x):=\sum_{i\in J} f_i(x)b_i  $, $x\in E$.  
Let $x_0\in E$ and  $E_2=\ker Df(x_0)$, $X_1:=\spanc(Df_i(x_0),\ {i\in J})$ and assume that:
\begin{enumerate}[label=\emph{(\Alph*)}]
    \item 
    $(Df_i(x_0))_{i\in J}$ is a basis for $X_1$, 
    \item $(Df_i(x_0))_{i\in J}$ is Besselian for $X_1$.
\end{enumerate}
Then $E=E_1 \oplus E_2$, where $E_1=\spanc (Df_i(x_0)^*,\ {i\in J}).$ 
\end{corollary}
\begin{proof}
The proof follows directly from James theorem applied to $X_1$ and Proposition  \ref{proposition:E1_representation}.
\end{proof}
 
\begin{remark}\label{remark:CRC_kernels} Assume that CRC holds for $(f_i)_{i\in J}$ at $x_0$ with  $V(x_{0})$ 
as in Definition \ref{definition:crc}. Then, by CRC \ref{item:CRC:3},
$$
\forall_{k\in J_2}\ \forall _{e\in E}\ z_{x}\circ(b^{*}_{k}(Df(x_{0}))(e)=b^{*}_{k}(Df(x)(e))
$$
and 
\begin{equation*}
     b_l^*(Df(x)(e))=\sum_{i\in J_2} \beta_{i} Df_i(x)(e), \quad l\in J\setminus J_2,\ x\in V(x_0).
\end{equation*}
Hence 
\begin{equation} 
\label{eq:ker}
e\in \ker Df(x)\begin{array}[t]{l}
\Leftrightarrow\ \forall_{k\in J}\ b_{k}^{*}(Df(x)(e))=0\\
\Leftrightarrow\ \forall_{k\in J_2}\  z_{x}\left(b_{k}^{*}(Df(x_{0})(e))\right)=0\\
\Leftrightarrow\ \forall_{k\in J}\ b_{k}^{*}(Df(x_{0})(e))=0\
\Leftrightarrow e\in \ker Df(x_{0}).
\end{array}
\end{equation}

\end{remark}

 For any $f:E\rightarrow F$, and $J_1\subset\mathbb{N}$, we have
$f(x)=\sum_{i\in J_1} b_i^*(f(x))b_i$, $x\in E$, where
$f_{i}(x)=b_i^*(f(x))=0$ for $i\in \mathbb{N}\setminus J_1$. Consequently, for any $e\in E$,
\begin{equation} 
\label{eq_derivative}
Df(x)(e)=\sum_{j\in J_1}b_j^*(Df(x)e)b_i=\sum_{j\in J_1}(Df_{j}(x)e)b_j,
\end{equation}
where, by Fact \ref{fact:representation_bstar}, $Df_{j}(x)e=b_j^*(Df(x)e)$, $j\in J_1$.

In the lemma below we investigate the coefficients
$Df_{j}(x)e=b_j^*(Df(x)e)$, $j\in J_1$, $e\in E$, in a neighbourhood  of $ x_{0}$ at which CRC holds.

\begin{lemma}\label{lemma:basis2}
	Let $E$ be a 
	Banach space, $F$ be a Hilbert space and  $(b_i)_{\in \mathbb{N}}$ be a Besselian basis of $F$. Let  
$(f_{i})_{i\in J}:\ E \rightarrow Y(J)$, $J\subset \mathbb{N}$ be of class $C^1$, $f:\ E \rightarrow F$, $f(x):=\sum_{i\in J} f_i(x)b_i  $, $x\in E$.  
		Assume that CRC+ holds for $(f_i)_{i\in J_1}$ at $x_0$ with $J_2\subset J_1$ and neighbourhood $V(x_0)\subset U$.

	 Let $x\in V(x_0)$, $e\in E$. 
	 Then  there exist scalars $\beta _{j}=\beta_{j}(e)$, $j
	  \in J_{2}$ depending on $e$ but not on $x$ such that
		\begin{equation}
		\label{coeficients}
		    Df_i(x)(e)=\sum_{j\in J_2}\beta_j(e)w_j^i(x), \quad i\in J_1,
		\end{equation}
		where
		\begin{equation*}
		w_j^i(x):=b_i^*(Df(x) (D f_{j}(x_0)^*))
,\quad i,j\in J_2
		\end{equation*}		and 
		$D f_{i}(x_0)^*\in E$, $i\in J_2,$	(see Proposition \ref{fact:existence_dial_basis}) are such that
		\begin{equation}
		\label{eq:zero:one}
		w_{j}^{i}(x_{0})=D f_{i}(x_0)^*(D f_{j}(x_0))=\left\{\begin{array}{ll}
		1 & \text{if } i=j,\\
		0 & \text{if } i\neq j
		\end{array}		\quad i,j \in J_2.
		\right.
		\end{equation}
\end{lemma}

 \begin{proof}
	The existence of $Df_i(x_0)^*\in E^{**}$, $i\in J_2$ is ensured by Proposition \ref{fact:existence_dial_basis}. By \ref{assumption:Cplus} of CRC+, 
	$Df_i(x_0)^*\in E$, $i\in J_2$. 
	Since, by \ref{assumption:Bplus} of CRC+,  $D f_{i}(x_0)$, $i\in J_2$ forms a basis of $\spanc ( Df_i(x_0),\ i \in J_1)  $ by Proposition \ref{fact:existence_dial_basis}, 
	$D f_{i}(x_0)^*$, $i\in J_2$ forms a basis of $(\spanc ( Df_i(x_0),\ i \in J_1))^*  $.
	
		
By Proposition \ref{proposition:E1_representation},
		 \begin{align}\label{eq:coefficients}
		 \begin{aligned}
		      & e=e_{1}+e_{2},\ e_{2}\in\text{ker } Df(x_{0}),\\
		      & e_{1}=\sum_{k\in J_{2}}\beta_{k}(e)Df_{k}^{*}(x_{0}),\ \beta_{k}(e)\in\mathbb{R},\ k\in J_2.
		 \end{aligned}
		 \end{align}

	By \ref{item:CRC:3plus} of CRC+  and  \eqref{eq:ker}, $e_{2}\in\text{ker } Df(x)$ for $x\in V(x_{0})$ and for  $i\in J_1$ we have
\begin{align}
\begin{aligned}
\label{formula:Dfi}
    		 Df_i(x)(e)&=b_i^*(Df(x)(e_1+e_2))=b_i^*(Df(x)\sum_{k\in J_{2}}\beta_{k}(e)Df_{k}^{*}(x_{0}))\\
    		 &=\sum_{k\in J_{2}}\beta_{k}(e)b_{i}^{*}(Df(x)Df_{k}^{*}(x_{0})).
    		 \end{aligned}
\end{align}
	
	   \end{proof}

		 \begin{remark}
		 In other words, \eqref{coeficients} means that for $x\in V(x_{0})$ and $i\in J_{2}$
			\begin{equation}
		\label{coeficients1}
		    Df_i(x)(e)=\sum_{j\in J_2}\beta_j(e)w_j^{i}(x)=\sum_{j\in J_2}\beta_j(e)b_i^*(Df(x) (D f_{j}(x_0)^*)).
		\end{equation}
		In particular, for $x=x_{0}$ and $i\in J_{2}$,
		\begin{equation}
		\label{coeficients2}
		    Df_i(x_{0})(e)=\sum_{j\in J_2}\beta_j(e)w_j^{i}(x_{0})=\sum_{j\in J_2}\beta_j(e)b_i^*(Df(x_{0}) (D f_{j}(x_0)^*))=\beta_{i}(e).
		\end{equation}

			 \end{remark}

	\begin{remark}
	By Proposition \ref{fact:existence_dial_basis}, vectors $w_j(x_0)$, $j\in J_2$ are such that 
	\begin{equation*}
	     (w_{j}^{i}(x_0))_{i\in J_2}=  \left\{ \begin{array}{rc}
	        1, & i=j, \\
	        0, & i\neq j.
	    \end{array} \right. 
	\end{equation*}

 Observe that by \eqref{formula:Dfi}, for $x=x_{0}$ and $j\in J_{1}\setminus J_{2}$
$$
\sum_{k\in J_{2}}\beta_{k}(e)b_{j}^{*}(Df(x_{0})Df_{k}^{*}(x_{0}))=0
$$
 because $Df(x_{0})Df_{k}^{*}(x_{0})=\sum_{i\in J_1} Df_i(x_0)(Df_k^*(x_0))b_i=b_{k}$ for $k\in J_{2}$ and, by the definition of $b^{*}_{j}$, 
		$$
		b^{*}_{j}(b_{k})=0\ \ \text{for}\ \ j\in J_{1}\setminus J_{2}.
		$$
		Hence 
			\begin{equation}
		\label{coeficients2aa}
		    Df(x_{0})(e)=\sum_{j\in J_1}Df_{j}(x_{0})(e)b_{j}
		    =\sum_{j\in J_2}\beta_j(e)b_{j}.
		\end{equation}
	\end{remark}

In  the following proposition we prove that CRC+
ensure that the mapping defined by \eqref{isomorphism} is an isomorphism. This proposition together with the Rank Theorem (Theorem \ref{theorem:rank}) will allow us to prove Proposition \ref{proposition:functional_dependence} which is crucial in the proof of the main result in Section \ref{section:main_result}.

\begin{proposition}\label{proposition:CRC_and_isomorphism}
	Let $E$ be a 
	Banach space, $F$ be a Hilbert space and assume that $(b_i)_{i\in \mathbb{N}}$ is a Besselian basis of $F$ and $J_1\subset \mathbb{N}$.  
	Let $(f_i)_{i\in J_1}:\ U \rightarrow Y(J_1)$,  
	$U\subset E$ open,  be  a $C^{1}$ mapping.
	Assume that CRC+ holds for $(f_i)_{i\in J_1}$ 
	at $x_0$ with a neighbourhood $V(x_0)$ with index set $J_2\subset J_1$.	
	
	
 	
 	Then 
	\begin{equation} 
	\label{isomorphism}
    t_x:=((Df_i(x))_{i\in J_1})\big|_{E_1} : E_1 \rightarrow   (Df_i(x))_{i\in J_1}(E),\quad \forall x\in V(x_0)
	\end{equation}
	is an isomorphism, where $E_1:=X_1^*=\spanc (Df_i(x_0)^*,\ i\in J_2)$.
\end{proposition}

\begin{proof}

By \ref{item:CRC:1plus} of CRC+, for any $x\in V(x_0)$ for any $l\in J_1\setminus J_2$, $Df_l(x)(e_1)$ can be expressed by $Df_i(x)(e_1)$, $i\in J_2$, i.e., there exists scalars $\alpha _i^l (x)$, $i\in J_2$, $l\in J_1\setminus J_2$ such that
\begin{align*}
    Df_l(x)(e_1)=\sum_{i\in J_2}\alpha _i^l (x)Df_i(x)(e_1).
\end{align*}

Since $e_1\in E_1$,  by Lemma \ref{lemma:basis2} we have
    \begin{align*}
    Df_i(x)(e_1) =  \sum_{j\in J_2} \beta_j(e_1) Df_i(x) Df_j(x_0)^*, \quad i\in J_2.
\end{align*}
 Since $(f_i)_{i\in J_1}$ is a $C^{1}$ mapping, $t_{x}$ is a continuous (linear) mapping. Now we show that $t_{x}$, $x\in V(x_{0})$, is a bijection.

1. Step: injectivity. 
Let $x\in V(x_0)$.
Let us take $x_1, x_2 \in E_1$. 
 Suppose that $(Df_i(x))_{i\in J_1}(x_1)=(Df_i(x))_{i\in J_1}(x_2)$. Then $(Df_i(x))_{i\in J_1}(x_1-x_2)=0$ and by Remark \ref{remark:CRC_kernels}, $x_1-x_2\in \ker (Df_i(x)_{i\in J_2})=\ker (Df_i(x_0)_{i\in J_1})=\bigcap_{i\in J_1} Df_i(x_0)$. On the other hand $x_1-x_2\in E_1$, hence $x_1=x_2$.

2. Step: surjectivity. 
Let $e\in E$ and $x\in V(x_0)$. By \ref{item:CRC:1plus} of CRC+, 
\begin{align*}
     Df_l(x)(e)=\sum_{i\in J_2} \alpha_i^l(x)Df_i(x)(e),\quad l\in J_1\setminus J_2.
\end{align*}
For $i\in J_2$
\begin{align}
\label{basic1}
\begin{aligned}
    Df_i(x)(e)&=\sum_{j\in J_2} \beta_j(e) (Df_i(x)Df_j(x_0)^*)\\
    &=Df_i(x)(\sum_{j\in J_2} \beta_j(e) Df_j(x_0)^*)
    \end{aligned}
\end{align}
and for $l\in J_1\setminus J_2$
\begin{align}
\label{nonbasic1}
\begin{aligned}
    Df_l(x)(e)& = \sum_{i\in J_2} \alpha_i^l(x)Df_i(x)(e)\\
    &= \sum_{i\in J_2} \alpha_i^l(x)Df_i(x) \left(\sum_{j\in J_2} \beta_j(e) Df_j(x_0)^*\right)\\
    & = \sum_{i\in J_2} \alpha_i^l(x)Df_i(x) \left(e_1 \right)\\
    &=Df_l(x)\left(e_1 \right),
     \end{aligned}
\end{align}
where $e_1=\sum_{j\in J_2} \beta_j(e)Df_j(x_0)^*\in E_1$ (hence  $Df_i(x)(e)=Df_i(x)(e_1)$, $i\in J_2$).
Since $Df_i(x)(e)=Df_i(x)(e_1)$, $i\in J_2$ and $Df_l(x)(e)=Df_l(x)(e_1)$, $l\in J_1\setminus J_2$ we obtain that $(Df_i(x))_{i\in J_1}(e)=(Df_i(x))_{i\in J_1}(e_1)$. 

\end{proof}


\section{ CRC+ and functional dependence} \label{section:CRCandfunctional}

 In this section we use CRC+ to prove Proposition \ref{proposition:functional_dependence} which provides conditions for the  functional dependence in the form of formula \eqref{fun_dep} and is based on  Proposition \ref{proposition:E1_representation}, Proposition \ref{proposition:CRC_and_isomorphism}, and  Rank Theorem (Theorem \ref{theorem:rank}). Proposition \ref{proposition:functional_dependence} will be used in the proof of the main result, Theorem
\ref{theorem:tangent_cone} of Section \ref{section:main_result}.

For convenience of the Reader we start by recalling the rank and the local representation theorems.

\begin{theorem}(Rank Theorem, see \cite[Theorem 2.5.15]{manifolds_tensor_vol2} )\label{theorem:rank}
Let $E,\ Y$ be
Banach spaces. 
	Let $x_0\in U$, where $U$ is an open subset of $E$ and $f:\ U\rightarrow Y$ be of class $C^1.$ 
	
	Assume that $Df(x_0)$ has closed split image $Y_1$ with closed component $Y_2$ and split kernel $E_2$ with closed component $E_1$ and that for all $x$ in some neighbourhood of $x_0$, $Df(x)|E_1:\ E_1 \rightarrow Df(x)(E)$ is an isomorphism.

	Then there exist open sets $U_1\subset Y_1\oplus E_2$, $U_2\subset E$, $V_1\subset Y$, $V_2\subset Y_1\oplus E_2$ and diffeomorphisms of class $C^1$, $\varphi:\ V_1\rightarrow V_2$ and $\psi:\ U_1\rightarrow U_2$,  $x_0=(x_{01},x_{02})\in U_2\subset U\subset E_1\oplus E_2$, i.e. $x_{01}\in E_1$, $x_{02}\in E_2$, $f(x_0)\in V_1$ satisfying 
	\begin{equation*}
		(\varphi \circ f \circ \psi )(w,e)=(w,0),\quad \text{where}\ w\in Y_1,\ e\in E_2
	\end{equation*} for all $(w,e)\in U_1$.
\end{theorem}

\begin{theorem}\label{theorem:local_representation}(\cite[Theorem 2.5.14]{manifolds_tensor_vol2} Local Representation Theorem)
 Let $E,\ Y$ be 
Banach spaces. 	Let $f:\ U \rightarrow Y$ be of class $C^r$, $r\geq1$ in a neighbourhood of $x_0\in U$, $U\subset E$ open set. Let $Y_1$ be closed split image of $Df(x_0)$ with closed complement $Y_2$.
	Suppose that  $Df(x_0)$ has split kernel $E_2=\ker Df(x_0)$ with closed complement $E_1$. 
	Then there are
	open sets $U_1\subset U \subset E_1\oplus E_2 $ and $U_2 \subset Y_1 \oplus E_2$, $x_0\in U_2$  and a $C^r$ diffeomorphism 
	$\psi:\ U_2\rightarrow U_1 $ such that $(f\circ \psi)(u,v)=(u,\eta(u,v))$ for any  $(u,v)\in U_1$, where 
	$u\in E_1$, $v\in E_2$ and  $\eta:\ U_2\rightarrow E_2$ is a $C^r$ map satisfying $D\eta(\psi^{-1}(x_0))=0$.
\end{theorem}

Theorem \ref{theorem:rank} and Theorem \ref{theorem:local_representation} allow to prove the following functional dependence result.

\begin{proposition}\label{proposition:functional_dependence}
Let $E$ be a 
	Banach space, $F$ be a Hilbert space with Besselian and Hilbertian basis $(b_i)_{i\in \mathbb{N}}$.
	Let $(f_i)_{i\in J_1}:\ U \rightarrow Y(J_1)$, 
	$U\subset E$ open,  be  a $C^{1}$ mapping.
	Let 
$E_2=\text{ker}\, (Df_i(x_0)_{i\in J_1})$.
Assume that CRC+ holds for $(f_i)_{i\in J_1}$, at $x_0$ with the index set $J_2\subset J_1$ and a neighbourhood $V(x_0)$.

Then there exist functions $h_l:\ Y(J_2) \rightarrow \mathbb{R}$, $l\in J_1\setminus J_2$  of class $C^1$ such that
\begin{equation}
\label{fun_dep}
    f_l(x)=h_l((f_i(x))_{i\in J_2})\quad \text{for all } x \text{ in some neighbourhood of } x_0.
\end{equation}
\end{proposition}
\begin{proof}
By \ref{item:CRC:2plus} of CRC+, $Y_1:=(Df_i(x_0)(E))_{i\in J_2}$ is a closed subset of $Y(J_2)$. 
By Proposition \ref{proposition:E1_representation},  applied to $(f_{i})_{i\in J_2}:\ E \rightarrow Y(J_2)$, $(Df_i)_{i\in J_2}(x_0)$ has a split kernel $E_2=\ker ((Df_i(x_0))_{i\in J_2})=\ker ((Df_i(x_0))_{i\in J_1})$ with closed complement $E_1=\spanc (Df_i(x_0)^*,i\in J_2)=\spanc (Df_i(x_0)^*,i\in J_1)$, $E=E_1 \oplus E_2$.
By Proposition \ref{proposition:CRC_and_isomorphism}, $
 (Df_i(x)_{i\in J_1})_{|E_1}:\ E_1\rightarrow (Df_i(x)_{i\in J_1})(E)$, $x\in V(x_0)$ is an isomorphism. By \ref{item:CRC:1plus} of CRC+,
\begin{equation*}
    Df_i(x_0)(e)= \sum_{j\in J_2}( \beta_{j}^i Df_j(x_0)(e) ),\quad i\in J_1\setminus J_2,
\end{equation*}
 where $\beta_{j}^i\in \mathbb{R}$, $i\in J_1 \setminus J_2$, $j\in J_2$.
 
 Since $Y_1$ is closed, by Proposition 
 \ref{proposition_splits},  
  $\spanc (b_i, i\in J_1)$
 splits, i.e., there exists $Y_2$ such that 
 $\spanc (b_i, i\in J_1)=Y_1\oplus Y_2$.

 By Rank Theorem \ref{theorem:rank}, there are
	open sets $U_2\subset U \subset E_1\oplus E_2 $ and $U_1 \subset Y_1 \oplus E_2$, $x_0\in U_2$, $V_1\subset \spanc ( b_i,\ i\in J_1)$, $V_2\subset Y_1\oplus E_2$  and $C^1$ diffeomorphisms 
	$\psi:\ U_1\rightarrow U_2 $, $\varphi:\ V_1\rightarrow V_2$  such that for all $x\in U_2$ there exists $(q,e)\in U_1$ such that $x=\psi(q,e)$ and
	\begin{equation}\label{composition:not_dependent}
	    (f_i)_{\in J_1}\circ\psi (q,e)=\varphi^{-1}(q,0).
	\end{equation} 
	By Local Representation Theorem \ref{theorem:local_representation}, there exists a function $\eta:\ U_1\rightarrow Y_2$ of class $C^1$ such that for all $(q,e)\in U_1$ 
	\begin{equation}\label{composition:implicit_function}
	    (f_i)_{i\in J_1} \circ \psi (q,e)=(q,\eta(q,e)) \in Y(J_1)
	\end{equation}
	Let us put 
	$(\bar{f}_i)_{i\in J_1}(q):=(q,\eta(q,e))=\varphi^{-1}(q,0)$ for $(q,e)\in U_1$, where $\varphi^{-1}(q,0)=(\varphi_i^{-1}(q,0))_{i\in J_1}$. 
	
	Then, by \eqref{composition:not_dependent} and \eqref{composition:implicit_function} we have that
	\begin{align*}
	    & q_i:=\bar{f}_i(q)=\varphi_i^{-1}(q,0)=f_i\circ \psi (q,e)=f_i(x),\quad i\in J_2
\end{align*}
and for any $l\in J_1\setminus J_2$,
\begin{align*}
	    &f_l(x)=f_l\circ \psi (q,e)=\bar{f}_l(q)=f_l\circ \psi (q,e)\\
	    & = b_l^*(f\circ\psi (q,e)) ) = b_l^*(q,\eta(q,e))\\
	    &= \eta_\ell(q,e)= \eta_\ell(q,\bar{e}).
	\end{align*}
    Therefore,
	\begin{equation*}
	  \eta_\ell(q,\bar{e})  =:h_l((f_i(x))_{i\in J_2}),\quad l\in J_1\setminus J_2,
	\end{equation*}
	where $h_l:\ Y(J_2)\rightarrow \mathbb{R}$.
	
\end{proof}


\section{Relaxed Constant Rank Constraint Qualification Plus}\label{section:Abadie_condition}

 \begin{definition}
\label{rcrcq_condition} Let  $(g_{i})_{i\in I_0\cup I_1}:\ E \rightarrow Y(I_0\cup I_1)$, $I_0\cup I_1\subset \mathbb{N}$, $I_0\cap I_1=\emptyset$ be of class $C^1$. 
We say that Relaxed Constant Rank Constraint Qualification Plus (RCRCQ+ in short) holds for set $\mathcal{ F}$, given by \eqref{set:constraints}, at 
$x_0$ if 
there exists a neighbourhood $V(x_0)$ such that 
for all $J$, $I_{0}\subset J\subset I_0\cup I(x_0)$,  
CRC+ holds for $(g_i)_{i\in J}$ at $x_0$ with neighbourhood $V(x_0)$, i.e., 
 exists $J_2\subset J$  such that 
 \begin{enumerate}[label=\emph{\arabic*.}]
\item  $(Dg_i(x))_{i\in J_2}$ 
forms a Schauder basis for 
 $\spanc (D g_{i}(x), i \in J)$ for all $x\in V(x_0)$,
     \item for any $x\in V(x_0)$, there exists a topological isomorphism
	\begin{equation*}
	    z_x: \spanc (D g_{i}(x), i \in J) \rightarrow \spanc (D g_{i}(x_0), i \in J), 
	\end{equation*}
	such that $z_x(D g_{i}(x_0))= D g_{i}(x)$, $i \in J_2$,
	\end{enumerate}
	and additionally
\begin{enumerate}[label=\emph{\arabic*.}]
\setcounter{enumi}{2}
 \item $(Dg_i(x_0))_{i\in J_2}(E)$ is closed in $Y(J_2)$ defined in \eqref{def:Y(J)}, with $v_i=b_i$, $i\in J_2$,
    \item $(Dg_i(x_0))_{i\in J_2}$ forms a shrinking and boundedly-complete basis for $X_1:=\spanc (Dg_i(x_0),\ i\in J)$,
	    \item $(Dg_i(x_0))_{i\in J_2}$ is Besselian for $\spanc(Dg_i(x_0),\ i\in J)$,
	    \item\label{assumtion:RCRCQ6} $Dg_i(x_0)^*\in E$, $i\in J_2$.
	 \end{enumerate}
\end{definition}

\begin{remark}
Let us note that if there is no inequality constraints, i.e. $I_1=\emptyset$, then the Relaxed Constant Rank Constraint Qualification Plus for set $\mathcal{ F }$ at $x_0$ is equivalent to Constant Rank Condition Plus for $(g_i)_{i\in I_0}$ at $x_0$. 
\end{remark}

\begin{remark}
If $E$ is a reflexive space, then RCRCQ+ holds for set $\mathcal{ F}$ given by \eqref{set:constraints} takes the form:
there exists a neighbourhood $V(x_0)$ such that 
for all $J$, $I_{0}\subset J\subset I_0\cup I(x_0)$,  
 exists $J_2\subset J$  such that 
 \begin{enumerate}[label=\emph{\arabic*.}]
 \item  $(Dg_i(x))_{i\in J_2}$ 
forms a Schauder basis for 
 $\spanc (D g_{i}(x), i \in J)$ for all $x\in V(x_0)$,
     \item for any $x\in V(x_0)$, there exists a topological isomorphism
	\begin{equation*}
	    z_x: \spanc (D g_{i}(x), i \in J) \rightarrow \spanc (D g_{i}(x_0), i \in J), 
	\end{equation*}
	such that $z_x(D g_{i}(x_0))= D g_{i}(x)$, $i \in J_2$,
	\end{enumerate}
and
\begin{enumerate}[label=\emph{\arabic*.}]
\setcounter{enumi}{2}
\item $(Dg_i(x_0))_{i\in J_2}(E)$ is closed in $Y(J_2)$ defined in \eqref{def:Y(J)}, with $v_i=b_i$, $i\in J_2$,
	    \item $(Dg_i(x_0))_{i\in J_2}$ is Besselian for $\spanc(Dg_i(x_0),\ i\in J)$.
	 \end{enumerate}

\end{remark}

Let us note that if $J\subset \mathbb{N}$ is finite, $|J|=n$ and $g_i:\ E \rightarrow \mathbb{R}$, $i\in J$ are of class $C^1$ in some neighbourhood of $x_0$ then $((Dg_i(x_0))_{i\in J})(E)$ is closed in $\mathbb{R}^n$.

\section{Main result}\label{section:main_result}

The following condition will be used in our main theorem.

\begin{enumerate}[label=(H\arabic*)]
    \item \label{assumption:lemma}  For all $d\in \Gamma_\mathcal{ F}(x_0)$ and for any vector function $r:\ (0,1)\rightarrow E$ such that $\|r(t)\|t^{-1}\rightarrow 0$, as $t\downarrow 0$, there exists a number $\varepsilon_0>0$ such that
	\begin{equation*}
	g_i(x_0+td+r(t))\leq 0 \text{ for all } i\in  I_1 \setminus I(x_0,d) \text{ and for all } t\in (0,\varepsilon_0),
	\end{equation*} 
	where $I(x_0,d):=\{ i\in I(x_0)\mid \langle D g_i(x_0) \, , \,  d \rangle =0 \}$.
\end{enumerate}

Note that if for some $d\in \Gamma_\mathcal{ F}(x_0)$, set $I_1\setminus I(x_0,d)$ is finite, then the Condition \ref{assumption:lemma} is satisfied, cf. \cite[Lemma 6.4]{MR4223867}.
 
 The following examples illustrates condition \ref{assumption:lemma}.
 \begin{example}\label{example:assumption:lemma}
        Let $h:\ \mathbb{R}\rightarrow \hilbertX$, $\hilbertX$ is  sequence space $\ell_p$, $p\geq 1$, $h_i(x)=a_i x^2 + b_i x + c_i$, $i=1,\dots$ where $a_i\leq 0$, $i=1,\dots$ and
        \begin{equation*}
            c_i=\left\{ \begin{array}{rl}
                 0, & i - \text{even}, \\
                <0, & i - \text{odd}
            \end{array} \right., \quad i=1,2, \dots
        \end{equation*}
        Let $\mathcal{ F}:= \{ x \in \mathbb{R} \mid  h(x)\in \hilbertX_{-} \}$, where $h(x)=(h_i(x))_{i\in I_0\cup I_1}$ and $x_0=0$. 
        Then $I(x_0)=\{2k,\ k\in \mathbb{N}\}$. Assume that $b_i> 0$, $i=1,\ldots$.
        Then
        \begin{align*}
            &\Gamma_\mathcal{ F}(0)=\{ d\in \mathbb{R} \mid b_i \cdot d \leq 0,\ i\in I(x_0)\} = \mathbb{R}_{-}\\
            & \mathcal{ T}_\mathcal{ F}(0)=\{ d \in \mathbb{R} \mid \exists r(t),\ r(t)/t\rightarrow 0^+ \ \exists \varepsilon_0>0\ \forall t\in [0,\varepsilon_0) \quad  h(x_0+td+r(t))\in \hilbertX_{-}\}.
        \end{align*}
        Let us take $d=-1$. 
        We have 
        \begin{align*}
            &a_i(-t+r(t))^2+b_i(-t+r(t))+c_i\leq 0, \quad i=1,\ldots\\
            & \iff t^2 a_i(-1+\frac{r(t)}{t})^2+tb_i(-1+\frac{r(t)}{t})+c_i \leq 0, \quad i=1,\ldots.
        \end{align*}
        Since $a_i,c_i\leq 0$ and $b_i\geq 0$, $i=1,\dots$ we obtain that $-1 \in \mathcal{ T}_\mathcal{ F}(0)$. Since $\mathcal{ T}_\mathcal{ F}(0)\subset \Gamma_\mathcal{ F}(0)$ we obtain that Abadie condition holds, i.e., $\mathcal{ T}_\mathcal{ F}(0)= \Gamma_\mathcal{ F}(0)$. Let us also prove that $h_i$ also satisfy condition \ref{assumption:lemma}.
        Indeed, if $i\in I(0) \setminus I(0;d)$, we have $b_i \cdot d <0.$ So inequality     $$
        h_i(td+r(t)) = a_i(td+r(t))^2 + t b_i ( d + \frac{r(t)}{t}) <0
        $$
        is satisfied for any function $r(t)$ s.t. $r(t)/t\rightarrow 0^+$ for $t\in (0, \varepsilon_0)$ for some $\varepsilon_0$. 
\end{example}

Note that condition PMFCQ from \cite{MR3070104} is not satisfied for functions in Example \ref{example:assumption:lemma}.

In the following example RCRCQ+ and Abadie conditions are not satisfied. 
\begin{example}
Let $g_j:\ \mathbb{R}^2 \rightarrow \mathbb{R}$, $j\in \mathbb{N}$ be defined by
\begin{align*}
    & g_1(x_1,x_2)=-x_1,\\
    & g_2(x_1,x_2)=-x_2;\\
    &g_j(x_1,x_2):=jx_1x_2, \forall (x_1,x_2)\in \mathbb{R}^2,\quad \forall j=3,4,\dots.
\end{align*}
Functions $g_j:\ \mathbb{R}^2 \rightarrow \mathbb{R}$, $j\in \mathbb{N}$ are differentiable, and
\begin{align*}
\mathcal{ F}:&= \{  (x_1,x_2)\in \mathbb{R}^2\mid g_j(x_1,x_2)\leq 0,\  j\in \mathbb{N}\}\\
&=\{  (x_1,x_2)\in \mathbb{R}^2\mid  x_1=0, x_2\geq 0\} \cup \{  (x_1,x_2)\in \mathbb{R}^2\mid  x_1\geq 0, x_2= 0\}.
\end{align*}
    Then 
    \begin{equation*}
        \mathcal{ T}_\mathcal{ F}(0,0)=\{ (u_1,u_2)\in \mathbb{R}^2 \mid u_1\geq 0,\ u_2\geq 0,\ u_1\cdot u_2 =0 \}
    \end{equation*}
    and
    \begin{equation*}
        \Gamma_\mathcal{ F}(0,0)=\mathbb{R}_+^2.
    \end{equation*}
    Hence Abadie condition does not hold. Moreover,
    RCRCQ+ does not hold, since
    \begin{align*}
        (Dg_i(0,0))_{i\in \{1,3\}}\quad  \text{is not isomorphic to}  \quad (Dg_i(x_1,x_2))_{i\in \{1,3\}}
    \end{align*}
    for any $(x_1,x_2)$ s.t. $x_2\neq 0$. Note that for any $d\in \Gamma_\mathcal{ F}(0,0)$, $I_1 \setminus I((0,0),d)$ is a finite set, thus it is easy to check that condition \ref{assumption:lemma} is satisfied.
\end{example}

\begin{theorem}(Ljusternik Theorem)\label{theorem:Ljusternik}
	Let $X$ and $Y$ be Banach spaces, let $U$ be a
	neighborhood of a point $x_0\in X$, and let $f:\ U \rightarrow Y$ be a Fr\'echet differentiable mapping. Assume that $f$ is regular at $x_0$, i.e., that
	$\text{Im}\, Df(x_0)= Y$,
	and that its derivative is continuous at this point (in the uniform operator
	topology of the space $\Gamma(X, Y)$). Then the tangent space  $\mathcal{ T}_M(x_0)$ to the set
	\begin{equation*}
		M =
		\{x \in U \mid f(x) =
			f(x_0)\}
	\end{equation*}
	at the point $x_0$ coincides with the kernel of the operator $Df(x_0)$,
	\begin{equation}
	\label{lust1}
	\mathcal{ T}_M(x_0) =
	\ker\, Df(x_0).
	\end{equation}
	Moreover, if the assumptions of the theorem are satisfied, then there exist a
	neighborhood $U'\subset U$ of the point $x_0$, a number $K>0$, and a mapping
	$\xi \rightarrow x (\xi)$ of the set $U'$ into $X$ such that
	\begin{align}	\label{lust2}
	\begin{aligned}
& f(\xi + x(\xi))= f(x_0),\\
& \|x(\xi)\|\leq K \|f(\xi) - f(x_0)\|
		\end{aligned}
	\end{align}
	for all $\xi \in U'$.

\end{theorem}

Assertion \eqref{lust2} follows from \eqref{lust1}, see e.g. \cite{theory_of_external_problems_Ioffe}. The assertion \eqref{lust2} is sometimes called the {\em generalized Ljusternik Theorem}, see e.g. \cite{MR601755}

Now we are ready to establish our main theorem.
\begin{theorem}\label{theorem:tangent_cone} 
Let $E$ be a 
	Banach space, $F$ be a Hilbert space and assume that $(b_i)_{i\in \mathbb{N}}$ is a Besselian and Hilbertian basis of $F$. Let $\mathcal{ F}\subset E$ be given as in \eqref{set:constraints}.  

Assume that
\begin{enumerate} [label=\emph{(\roman*)}]
\item RCRCQ+ holds for $\mathcal{ F}$ at $x_0\in \mathcal{ F}$
\item condition \ref{assumption:lemma} is satisfied at $x_0$. 
\end{enumerate}

Then Abadie condition holds, i.e., $\Gamma_\mathcal{ F}(x_0)=\mathcal{ T}_\mathcal{ F}(x_0)$.
	
	Moreover, for each $d\in \mathcal{ T}_\mathcal{ F}(x_{0})$ there is a vector function $r:\ (0,1)\rightarrow E$, $\|r(t)\|/t\rightarrow 0$ when $t\downarrow 0$, such that for all $t$ sufficiently small \begin{equation} 
		\begin{array}{l}
		g_i(x_{0}+td+r(t))=0,\ i\in J(d),\\
		g_{\ell}(x_{0}+td+r(t))\le 0,\ \ell \in I_1\setminus J(d),
		\end{array} \quad J(d):=I_{0}\cup I(x_{0},d).
		\end{equation} 
	
 
%

\end{theorem}
\begin{proof}
	The inclusion $\mathcal{ T}_\mathcal{ F}(x_0)\subset \Gamma_\mathcal{ F}(x_0)$ is immediate. To see the converse, take any $d\in \Gamma_\mathcal{ F}(x_0)$, where by Proposition \ref{proposition:representation_linearized},
	\begin{align*} 
\Gamma_{\mathcal{ F}}(x_{0})
\subset \{d\in E \mid \langle Dg_{i}(x_{0})\mid d\rangle=0,\ i\in I_0,\ \langle Dg_{i}(x_{0})\mid d\rangle\le 0,\ \ i\in I(x_{0})\}.
\end{align*}
Recall that $I(x):=\{ i\in I_1 \mid g_i(x)=0 \}$.
	
		We start by  considering the case $J:=J(d)\neq  \emptyset$, $|J|=+\infty$.
	
	By RCRCQ+ of $\mathcal{ F}$ at $x_0$, 
	there exist $V(x_0)$ and $J_2\subset J$ such that $(Dg_i(x_0))_{i\in J_2}$ forms a boundedly-complete and shrinking basis for $\spanc (Dg_i(x_0),i\in J)$ 
	and there exists a topological isomorphism $z$
	\begin{equation*}
		z_{t,r}:\ \spanc ( Dg_i(x_0+td+r), i\in J )\rightarrow \spanc (D g_i(x_0), i\in J ),
	\end{equation*}
	for $(t,r)$ in some neighbourhood of $(0,0)\in \mathbb{R}\times E$ such that
	$z_{t,r}(Dg_i(x_0+td+r))=Dg_i(x_0), i \in J_2$ for all $(t,r)\in \mathbb{R}\times E$ such that $x_0+td+r\in V(x_0)$. 
	

	
	
	
	
	Let $f:E\rightarrow Y(J_2)$ be defined as $f(x):=(g_i(x))_{i\in J_2}$, where $Y(J_2)$ is defined by \eqref{def:Y(J)}. By \eqref{eq:coefficients}, the derivative $Df(x_0)$ is onto $Y(J_2)$.

    Let us define 	\begin{equation}\label{set:M}
	M := \{ x\in E \mid g_{i}(x)=g_i(x_0)=0,\ i\in J_2  \}.
	\end{equation}  
    By  Ljusternik Theorem  \ref{theorem:Ljusternik} applied to the set $M$,
    \begin{equation*}
        \ker Dg_i(x_0)_{i\in J_2} = \mathcal{ T}_M(x_0).
    \end{equation*}
     By applying Ljusternik Theorem \ref{theorem:Ljusternik} with $f$ at $x_0$, we obtain that $d\in \mathcal{ T}_{M}(x_0)$. 
    	\begin{enumerate}
 	    \item[Case 1]  
     If $J_2=J$, then $g_i(x_0+td+r(t))=0$, $i\in J$ for $t\in [0,\varepsilon]$, where $\varepsilon>0$ and $r(t)$ is given by Ljusternik Theorem.


	
	\item[Case 2]
	If $J_2\subsetneq J$ then, by Proposition \ref{proposition:functional_dependence}, applied to $g_i$, $i\in J$, there exist  functions $h_l$, $l\in J\setminus J_2$ of class $C^1$, such that
	\begin{equation}\label{eq:implicit}
	g_l(x_0+td+r)=h_l((g_i(x_0+td+r))_{i\in J_2}),
	\end{equation} 
	for $(t,r)$ in some neighbourhood of  $(0,0)\subset \mathbb{R}\times E$.
	
	Consider the system
	\begin{equation}\label{eq:system_1}
	g_i(x_0+td+r)=0,\quad i\in J
	\end{equation} 
	with respect to variables $t,r$.
	Let us note that  system \eqref{eq:system_1} is satisfied for $(t,r)=(0,0).$ 
	
	Obviously, by Proposition \ref{proposition:functional_dependence}, in some neighbourhood of $(0,0)$,  system \eqref{eq:system_1} is equivalent to
	\begin{equation}\label{eq:system_2}
	g_i(x_0+td+r)=0, \quad i\in J_2 
	\end{equation}
	with additional condition
	\begin{equation}\label{eq:dependence}
	g_l(x_0+td+r)=h_l((g_i(x_0+td+r))_{i\in J_2})=0,\ l\in J\setminus J_2.
	\end{equation}
	Note that $h_l((g_i(x_0))_{i\in J_2})=0$, $l\in J\setminus J_2$ since $g_l(x_0)=0=h_l((g_i(x_0))_{i\in J_2})=h_l(0)$.
	
	

 	\end{enumerate}
 In both cases there exist $\varepsilon>0$ and a function $r:\ [0,\varepsilon)\rightarrow E$, $\|r(t)\|t^{-1}\rightarrow 0$, $t\downarrow 0,$ such that 
	\begin{equation*}
    g_i(x_0+td+r(t))=0, \quad i\in J,
	\end{equation*}
	i.e., 
	\begin{equation}\label{set:Ma}
	d\in \mathcal{ T}_{\tilde{M}}(x_0),\quad  \text{where}\ \tilde{M} := \{ x\in E \mid g_{i}(x)=g_i(x_0)=0,\ i\in J  \}.
	\end{equation}  

	By condition \ref{assumption:lemma}, there exists $\varepsilon_0>0$ such that
	\begin{equation*}
	g_i(x_0+td+r(t))\leq 0 \text{ for all } i\in (I_0\cup I_1)\setminus I(x_0,d) \text{ and for all } t\in (0,\varepsilon_0),
	\end{equation*} 
	therefore
	\begin{equation}
	x_0+td+r(t)\in \mathcal{ F}\quad t\in [0,\min\{\varepsilon_0,\varepsilon\}].
	\end{equation}
	Thus, $d\in \mathcal{ T}_\mathcal{ F}(x_0)$.
	
	Now, let us consider the  case $J=\emptyset$ (i.e. the case when both $I_0=\emptyset$ and $I(x_0,d)=\emptyset$). Then, by condition \ref{assumption:lemma}, for any vector function $r:\ (0,1)\rightarrow E$, $\|r(t)\|/t\rightarrow 0$ when $t\downarrow 0$ there exists $\varepsilon>0$ such that
	\begin{equation}
	x_0+td+r(t)\in \mathcal{ F}\quad t\in [0,\varepsilon],
	\end{equation}
	i.e.,  $d\in \mathcal{ T}_\mathcal{ F}(x_0)$. 
	%
		
\end{proof} 

\section{RCRCQ+ and Lagrange multipliers}\label{section:lagrange_multipliers}

In this section, using \cite{MR4159570} and RCRCQ+ condition we will prove non-emptiness of the Lagrange multipliers set. 

The \textit{Lagrange function} or \textit{Lagrangian} corresponding to problem \eqref{problem:P} is a function $L:\ E \times F\rightarrow \mathbb{R}$
\begin{equation*}
    L(x,\lambda):=f(x)+\langle \lambda \mid G(x) \rangle \quad x\in E,\ \lambda \in F.
\end{equation*}
\begin{definition}
Let $Q\subset F$ be a closed, convex set. The normal cone to $Q$ at $\bar{y} \in F$ is the set
\begin{equation*}
    \coneN_Q(\bar{y})= \{ y\in F \mid \langle y \mid k-\bar{y}  \rangle \leq 0 \ \forall k\in Q  \}.
\end{equation*}
\end{definition}

\begin{definition}
A feasible point $\bar{x}\in \mathcal{ F}$ of \eqref{problem:P} is called a KKT point if 
 there exists $\bar{\lambda}\in \mathcal{ N}_{K}(G(\bar{x}))$ such that
\begin{equation*}
    D_{\bar{x}}L(\bar{x},\bar{\lambda})=0,
\end{equation*}
where $K$ is defined by \eqref{cone:K}. 
In this case $\bar{\lambda}$ is called \textit{Lagrange multiplier} of \eqref{problem:P} at $\bar{x}\in \mathcal{ F}$
\end{definition}

Following \cite{MR4159570,MR3204130,Kurcyusz1976} let us define set corresponding to KKT points,
\begin{equation*}
    \setM(x_0,0):= DG(x_0)^* \coneN_{K}(G(x_0)).\footnote{Here $DG(x_0)^*$ is an adjoint operator to $DG(x_0)$}
\end{equation*}


\begin{proposition}
Let $x_0\in \mathcal{F}$. Assume that CRC+ holds for $(g_i)_{i \in I_0\cup I(x_0)}$ at $x_0$ with index set $J_2= I_0\cup I(x_0)$. 
Assume that $\operatorname{span} \{ Dg_i(x_0), i\in J_2 \}$ is closed. 
Then $\setM(x_0,0)$ is weakly*-closed.
\end{proposition}
\begin{proof}

Let $\langle \cdot , \cdot  \rangle_E:\ E\times E^* \rightarrow \mathbb{R}$ be a duality mapping for the pair $E,E^*$.  
Observe, that for any $x\in E$ and for any $y\in F$
\begin{equation*}
    \langle x , DG^*(x_0)y \rangle_E = \langle DG(x_0)x , y \rangle_F,
\end{equation*}
where $DG^*(x_0):\ F \rightarrow E^*$ denotes the adjoint operator. 
Therefore, for any $x\in E$ and any $l\in \mathbb{N}$,
\begin{equation}\label{eq:suml}
    \langle x , DG^*(x_0)b_l^* \rangle_E = \langle DG(x_0)x , b_l^* \rangle_F
    = \langle \sum_{i\in \mathbb{N}} b_i^*(DG(x_0)x)b_i  , b_l^* \rangle_F  = Dg_l(x_0)x.
\end{equation}
    Let $(z_n)$ be a sequence in $ \setM(x_0,0)\subset E^*$ weakly*-converging to some $z_0\in E^*$. We want to show that $z_0\in \setM(x_0,0)$, i.e., there exists $y_0\in \coneN_K(G(x_0))$ such that $z_0=DG^*(x_0)(y_0)$.
    
    There exist $y_n \in \coneN_K(G(x_0))\subset F$, $y_n=\sum_{i\in \mathbb{N}} b_i^{**}(y_n)b_i^* $, $n\in \mathbb{N}$ such that $z_n=DG^*(x_0)(y_n)$, $n\in \mathbb{N}$. Therefore, for each $n\in \mathbb{N}$
    \begin{equation*}
        \forall k \in K \quad \langle y_n \mid k - G(x_0) \rangle \leq 0
    \end{equation*}
    and in consequence
        \begin{equation*}
        \forall k \in K \quad \langle \sum_{i\in I_0\cup I_1} b_i^{**}(y_n)b_i^*  \mid k - G(x_0) \rangle \leq 0,
    \end{equation*}
i.e., according to \eqref{cone:K}
\begin{align*}
            &\forall k=\sum_{j\in I_0\cup I_1} \alpha_j b_j \in K,\ \alpha_j=0,\ j\in I_0,\  \alpha_j\leq 0, j\in I_1\\
            &\langle \sum_{i\in I_0\cup I_1} b_i^{**}(y_n)b_i^*  \mid \sum_{j\in I_0\cup I_1} (\alpha_j - b_j^*(G(x_0))) b_j \rangle \leq 0,
\end{align*}
and equivalently
\begin{align*}
              &\forall k=\sum_{j\in I_0\cup I_1} \alpha_j b_j \in K,\ \alpha_j=0,\ j\in I_0,\  \alpha_j\leq 0, j\in I_1\\
              &\sum_{i\in I_0\cup I_1} b_i^{**}(y_n) (\alpha_i - b_i^*(G(x_0))) \leq 0.
\end{align*}
Take any fixed $l\in I_0\cup I_1$.
\begin{enumerate}
    \item[Case 1] $b_l^{*}(G(x_0))<0$, i.e. $l\in I_1 \setminus I(x_0)$. By taking $\alpha_i=b_i^*(G(x_0))=0 $, $i\in I_0$, $\alpha_i=b_i^*(G(x_0))\leq 0$, $i\in I_1\setminus \{ l \}$ and $\alpha_l=0$ we deduce
\begin{equation*}
    b_l^{**}(y_n) (-b_l^*(G(x_0))) \leq 0\quad \forall n\in \mathbb{N}.
\end{equation*}
Therefore, $b_l^{**}(y_n)\leq 0$, $n\in \mathbb{N}$. 
On the other hand, by taking $\alpha_i=b_i^*(G(x_0))=0 $, $i\in I_0$, $\alpha_i=b_i^*(G(x_0))\leq 0$, $i\in I_1\setminus \{ l \}$ and $\alpha_l=2\cdot b_l^{*}(G(x_0))\leq 0$ we obtain
\begin{equation*}
     b_l^{**}(y_n) b_l^*(G(x_0)) \leq 0\quad \forall n\in \mathbb{N},
\end{equation*}
i.e., $b_l^{**}(y_n)\geq 0$, $n\in \mathbb{N}$. In conclusion, $b_l^{**}(y_n)= 0$, $n\in \mathbb{N}$.
    \item[Case 2] $b_l^{*}(G(x_0))=0$ and $l\in I(x_0)$. By taking $\alpha_i=b_i^*(G(x_0))=0 $, $i\in I_0$, $\alpha_i=b_i^*(G(x_0))\leq 0$, $i\in I_1\setminus \{ l \}$ and $\alpha_l=-1$ we obtain
\begin{equation*}
    b_l^{**}(y_n) \cdot (-1) \leq 0\quad \forall n\in \mathbb{N},
\end{equation*}
i.e. $b_l^{**}(y_n)\geq 0$.
\item[Case 3] $b_l^{*}(G(x_0))=0$ and $l\in I_0$. Then $b_l^{**}(y_n)\in \mathbb{R}$, $n\in \mathbb{N}$.
\end{enumerate}
Therefore,
\begin{align*}
    \coneN_K(G(x_0))= \{ y \in F \mid & y=\sum_{i\in I_0\cup I_1} b_i^{**}(y)b_i^*,\ b_i^{**}(y)=0, i\in I_1\setminus I(x_0),\\
    &b_i^{**}(y)\geq 0,\ i\in I(x_0),\ b_i^{**}(y)\in \mathbb{R},\ i\in I_0 \}.
\end{align*}

In conclusion $y_n=\sum_{i\in I_0\cup I(x_0)}b_i^{**}(y_n)b_i^*$, where $b_i^{**}(y_n)=0$, $i\in I_1\setminus I(x_0)$, $b_i^{**}(y_n)\geq 0$, $i\in  I(x_0)$, $b_i^{**}(y_n)\in \mathbb{R}$, $i\in I_0$ for any $n\in \mathbb{N}$.  Moreover, $z_n=DG^*(x_0)(y_n)=DG^*(x_0)(\sum_{i\in I_0\cup I(x_0)}b_i^{**}(y_n)b_i^*)$, $n\in \mathbb{N}$.


Let $a_i^n:=b_i^{**}(y_n)$, $i\in I_0\cup I_1$, $n\in \mathbb{N}$. 
Since $(z_n)$ is a sequence in $ \setM(x_0,0)$ weakly*-converging to some $z_0\in E^*$,
by \eqref{eq:suml}, for any $n\in \mathbb{N}$
\begin{align}\label{xz_rep1}
\begin{aligned}
\forall x \in E \quad 
    &
    \langle x , z_n \rangle_E=  \langle x , DG^*(x_0)(\sum_{i\in I_0\cup I(x_0)} a_i^n b_i^*) \rangle_E \\
    &=
    \sum_{i\in I_0\cup I(x_0)} a_i^n\langle x , DG^*(x_0)( b_i^*) \rangle_E\\
    &=
    \sum_{i\in I_0\cup I(x_0)} a_i^n\langle DG(x_0)x , b_i^* \rangle_F\\
    &=
    \sum_{i\in I_0\cup I(x_0)} a_i^n\langle \sum_{j\in I_0\cup I(x_0)} b_j^*(DG(x_0)(x))b_j , b_i^* \rangle_F\\
    &=
    \sum_{i\in I_0\cup I(x_0)} a_i^n\langle \sum_{j\in I_0\cup I(x_0)}b_j^*(DG(x_0)(x))b_j , b_i^* \rangle_F\\
    &=
    \sum_{i\in I_0 \cup I(x_0)} a_i^n Dg_i(x_0) ( x)\\
    &= 
    \langle x , \sum_{i\in I_0 \cup I(x_0)} a_i^n Dg_i(x_0) \rangle_E .
    \end{aligned}
\end{align}
By \eqref{xz_rep1}, $z_n=\sum_{i\in I_0 \cup I(x_0)} a_i^n Dg_i(x_0)\in X_1=\spanc( Dg_i(x_0), i\in I_0\cup I(x_0))$
and 
\begin{equation}\label{xz_rep2}
   \forall x\in E\quad  \langle x , z_0 \rangle_E=\lim_{n\rightarrow +\infty} 
    \langle x , z_n \rangle_E = \lim_{n\rightarrow +\infty} \langle x , \sum_{i\in I_0 \cup I(x_0)} a_i^n Dg_i(x_0) \rangle_E.
\end{equation}


Since $z_n$ converges weakly* in reflexive $X_1$, it converges to $z_0\in X_1$. Since $\operatorname{span} ( Dg_i(x_0),\ i\in J_2 ) =\spanc ( Dg_i(x_0),\ i\in J_2 )=X_1$,  $z_0=\sum_{i\in J_2} \tilde{\beta}_i Dg_i(x_0)$ for some $\tilde{\beta}_i\in \mathbb{R}$, $i\in J_2$.

By assumption \ref{item:CRC:1plus} of CRC+ with $J_1=I_0\cup I(x_0)$, $(Dg_k(x_0))_{k\in J_2}$ is a basis of $X_1=\spanc ( Dg_i(x_0),\ i\in I_0\cup I(x_0) )=\spanc ( Dg_i(x_0),\ i\in J_2 )$, hence, by Proposition \ref{fact:existence_dial_basis}, there exists $Dg_k^*(x_0)\in X_1^*$, $k\in J_2$ such that
\begin{equation*}
    Dg_k^*(x_0)(Dg_i(x_0))=\delta_{ik}\quad i,k \in J_2.
\end{equation*}
By \ref{assumption:Aplus} of CRC+, $(Dg_k^*(x_0))_{k\in J_2}$ is a basis for $\spanc ( Dg_i^*(x_0),\ i\in J_2 )$.

Let $\tilde{x}=Dg_j^*(x_0)\in X_1^*\subset  E^{**}$,  $j\in J_2$. By assumption \ref{assumption:Cplus} of CRC+ with $J_1=I_0\cup I(x_0)$, $\tilde{x}\in E$. By \eqref{xz_rep1} and \eqref{xz_rep2} with $x=\tilde{x}$ we have
\begin{equation*}
    \langle Dg_j^*(x_0),  \sum_{i\in I_0\cup I(x_0)} a_i^n Dg_i(x_0)  \rangle_E \rightarrow \langle Dg_j^*(x_0), \sum_{i\in J_2} \tilde{\beta}_i Dg_i(x_0)   \rangle_E,
\end{equation*}
is equivalent to
\begin{equation*}
    a_j^n\rightarrow \tilde{\beta}_j\quad j\in J_2 \quad \text{as}\ n\rightarrow +\infty.
\end{equation*}

Since 
$a_j^n=0$, $i\in I_1\setminus I(x_0)$, 
    $a_j^n\geq 0$, $i\in I(x_0)$, $a_j^n\in \mathbb{R}$,  $i\in I_0$ for any $n\in \mathbb{N}$ we have  
$\tilde{\beta}_i=0$, $i\in I_1\setminus I(x_0)$, 
    $\tilde{\beta}_i\geq 0$, $i\in I(x_0)$, $\tilde{\beta}_i\in \mathbb{R}$,  $i\in I_0$ for any $n\in \mathbb{N}$. In conclusion $z_0\in \setM(x_0,0)$.  

\end{proof}

By Proposition 5.6 of \cite{MR4159570} we immediately get the following result.
\begin{proposition}
Let $E$ be a 
	Banach space, $F$ be a Hilbert space and assume that $(b_i)_{i\in \mathbb{N}}$ is a Besselian and Hilbertian basis of $F$. Let $\mathcal{ F}\subset E$ be given as in \eqref{set:constraints}.  
Let $x_0\in \mathcal{ F}$ be a local minimizer of \eqref{problem:P}.
Assume RCRCQ+ holds for $\mathcal{ F}$ at $x_0\in \mathcal{ F}$ with a neigbourhood $V(x_0)$. 
Assume that assumption \ref{assumption:lemma} is satisfied at $x_0$. Assume that $\setM(x_0,0)$ is weakly* closed. Then the set of Lagrange multipliers at $x_0$ is nonempty.
\end{proposition}

\begin{proof}
    By Theorem \ref{theorem:tangent_cone}, Abadie condition holds for $\mathcal{ F}$ at $x_0$. The rest of the proof follows the lines of the proof of Proposition 5.6 of \cite{MR4159570}.
\end{proof}

\begin{remark}
Let us underline the fact that in both papers \cite{constant_rank_constraint_Andreani,MR4159570} in the definition of Abadie condition the set $\setM(x_0,0)$ is weakly* closed. 
\end{remark}

Let us consider now the case when no equality are present (i.e. $I_1=\emptyset$), i.e.,
\begin{equation*}       
\mathcal{ F}=\left\{\begin{array}{ll} x\in E \mid 
             g_i(x) = 0,&  i\in I_0
        \end{array}\right\}
\end{equation*}
Such problems has been considered in e.g. in 
\cite[Theorem 4.1]{MR4104521}. 
In this case we are also getting the existence of Lagrange multipliers (see Proposition \ref{proposition:lagrange_equalities}).  
By  Proposition \ref{proposition:E1_representation} we are getting split of $E$, which is included in assumption (B) of \cite[Theorem 4.1]{MR4104521} and by Proposition \ref{proposition:CRC_and_isomorphism} we are obtaining isomorphism of $((Dg_i(x))_{i\in I_0})|_{E_1}:\ E_1\rightarrow ((Dg_i(x))_{i\in I_0})(E)$, $x\in U(x_0)$, which is included in assumption (C) of \cite[Theorem 4.1]{MR4104521} (see Remark \ref{remark:assumptionC}).

\begin{proposition}\label{proposition:lagrange_equalities}
Let $E$ be a 
	Banach space, $F$ be a Hilbert space and assume that $(b_i)_{i\in \mathbb{N}}$ is a Besselian and Hilbertian basis of $F$. Let $\mathcal{ F}\subset E$ be given as in \eqref{set:constraints}, where $I_1=\emptyset$.  
Let $x_0\in \mathcal{ F}$ be a local minimizer of \eqref{problem:P}.
Assume CRC+ holds for $(g_i)_{i\in I_0}$ at $x_0\in \mathcal{ F}$ with a neigbourhood $V(x_0)$. 
Then the set of Lagrange multipliers at $x_0$ is nonempty.
\end{proposition}

\begin{proof}
    Let $x_0$ be a local minimizer of problem \eqref{problem:P0} with $I_1=\emptyset$.  The first-order necessary optimality condition is $Df_0(x_0)h=0$ for all $h\in \mathcal{ T}_\mathcal{ F}(x_0)$. 
   By the proof of Theorem \ref{theorem:tangent_cone} (see \eqref{set:Ma} with $J=I_0$),
    \begin{equation*}
                \ker Dg(x_0) = \ker (Dg_i(x_0))_{i\in I_0} = \mathcal{ T}_\mathcal{ F}(x_0).
    \end{equation*}
    Therefore $Df_0(x)h=0$ for any $h\in \ker DG(x_0)$, i.e., $Df_0(x)\in (\ker DG(x_0))^\perp$. By Proposition \ref{proposition:E1_representation},    
    $(\ker DG(x_0))^\perp=\spanc ( Dg_i(x_0)^*, i \in I_0 )$. Therefore $Df_0(x_0)\in \spanc ( Dg_i(x_0)^*, i \in I_0 )$, i.e., there exists $\lambda \in F$ such that $Df_0(x_0)=\langle DG(x_0)^* \mid \lambda \rangle = \lambda (DG(x_0))$. 
\end{proof}

\bibliographystyle{plain}
\bibliography{myarticlebibfile}   

\end{document}